\documentclass[11pt]{article}

\usepackage{amsfonts}
\usepackage{amsmath}
\usepackage{amsthm}
\usepackage{amssymb}
\usepackage{mathrsfs}
\usepackage{color}
\usepackage[top=2.5cm,bottom=3cm,right=2.5cm,left=2.5cm]{geometry}

\usepackage[utf8]{inputenc}
\usepackage[T1]{fontenc}

\newcommand{\N}{\mathbb{N}}

\newcommand{\R}{\mathbb{R}}

\theoremstyle{plain}

\theoremstyle{plain}
\newtheorem{theorem}{Theorem}[section]
\newtheorem{corollary}[theorem]{Corollary}
\newtheorem{lemma}[theorem]{Lemma}
\newtheorem{proposition}[theorem]{Proposition}

\theoremstyle{definition}
\newtheorem{definition}[theorem]{Definition}

\newtheorem{hyp}[theorem]{Hypothesis}

\title{Rigidity phenomenons for an infinite dimension diffusion operator and cases of near equality in the Bakry--Ledoux isoperimetric comparison Theorem.}
\author{\sc{Rapha{\"e}l Bouyrie } \\ 
\\ \textit{University of Marne--la--Vall\'ee, France}}
\date{}

\begin{document}
\maketitle

\begin{abstract}
We study rigidity phenomenons for infinite dimension diffusion operators of positive curvature using semigroup interpolations. In particular, for such diffusions, an analogous statement of Obata's theorem is established. Moreover, the same rigidity holds for the Bakry--Ledoux isoperimetric comparison Theorem - a result due to Franck Morgan. Recently, Mossel and Neeman have exploited the semigroup proof of Bakry and Ledoux to derive dimension free bounds for the Gaussian isoperimetry. We extend theirs arguments to obtain in particular new quantitative bounds on the spherical isoperimetric inequality in large dimension. 
\end{abstract}

\section{Introduction}

\subsection{General setting}
In this paper, we study rigidity phenomenons for spaces endowed with an infinite dimension diffusion operators. The natural setting consists of abstract Markov triples $(E, \mu, \Gamma)$, which belongs to the framework initiated by the seminal paper of Bakry and \'Emery  \cite{B-E}. For a complete account about Markov diffusion operator, we refer the reader to the recent monograph \cite{B-G-L}. This setting encircles as an important illustration the examples of weighted Riemannian manifolds with generalized Ricci curvature bounded from below and without any condition on the dimension, justifying the term ``infinite dimension''. 

In such framework, sharp geometric and functional inequalities has been established using the principle of monotonicity along the heat flow attached to the subsequent diffusion. In the case of Riemannian manifolds of finite dimension $n$ with positive Ricci curvature some rigidity occurs for inequalities such as spectral gap (a result due to Obata), Myers' maximal diameter theorem (known as Topogonov--Cheng theorem) or L\'evy--Gromov isoperimetric comparison theorem (see \cite{Bay}, \cite{Mi1}). That is, in those examples, a Riemannian manifold achieving equality is isometric to the model spaces, i.e. the Euclidean spheres. In two remarkable recent papers (\cite{C-M1}, \cite{C-M2}) Cavalletti and Mondino have proved analogous statements of these important results in the general framework of metric measure spaces with an appropriate notion of positive curvature and finite dimension. Theirs results goes by needle decompositions of metric measure spaces inspired by the work of Klartag (\cite{Kl}) in Riemannian geometry. Moreover similar rigidity phenomenons has been established by the authors, where in this more general setting the model spaces are spherical suspensions. 

The infinite dimensional case presents in various aspects less rigidity. The analogous models spaces are given by the Gaussian spaces of fixed variance. For example there is no longer boundedness on the diameter of manifolds (as immediately checked by considering Gaussian space). However, it is a natural guess that semigroup tools are well suited to investigate rigidity phenomenons in this setting : indeed since such inequalities are established by monotoncity along the heat flow, equality cases correspond to a constant evolution along such heat flow.

The aim of this paper is to investigate this idea more precisely. A common and classical scheme is to integrate in space and use a commutation property between the gradients and the underlying semigroups. Any function saturating these functional inequalities satisfies therefore equality in theses commutations. We will see that it implies that such function is necessary an eigenfunction of the underlying diffusion operator, a specific property of the Gaussian space. Before stating our results, we briefly recall the geometric and functional inequalities subsequent to this work.

\vspace{2mm}

The isoperimetric problem consists in minimizing the boundary of sets given a fixed volume. Isoperimetrical problems has a long history, and make sense in very general metric measure spaces $(E, d, \mu)$. Indeed, one only need a measure to define the ``volume'' and a distance to define the ``boundary''. In this context, denote the outer Minkowski content by
\[
\mu^+(A) = \liminf_{r \to 0} \frac{\mu(A_r) - \mu(A)}{r},
\]
with $A_r = \{ y \in E, d(x, A) \leq r \},$ where $d(x,A) = \mathrm{inf} \{d(x,y), \, y \in A \}$. Define then the isoperimetric profile of the set $(E, d, \mu)$ by
\[
\mathcal{I}_{(E, d, \mu)} (v) = \mathrm{inf}\{ \mu^{+}(A),  \, A \subset E, \mu(A) = v \}. 
\] 
The isoperimetric problem is two-fold : one aims to find the function $\mathcal{I}_{(E, d, \mu)}$ as well as describing the optimal sets. In general such a problem is very difficult to solve. Even the existence of optimal sets, called isoperimetric sets, is far from being guaranteed. Throughout this note, we will consider the isoperimetric problem over a Riemannian manifold $M$ with some probability measure $\mu$ of the form $e^{-\psi} d\mathrm{vol}$, where $\mathrm{vol}$ designs the canonical Riemannan volume. We will consider the usual distance induced by the metric $g$. As a consequence, $\mathcal{I}_{(M, d, \mu)}$ is defined on $[0,1]$ and symmetric with respect to $1/2$. 

The Euclidean spheres are among the few examples in which the isoperimetric problem is completely solved. This goes back to the early XXth century by L\'evy and Schmidt. The extremal sets are, perhaps unsurprisingly, the spherical caps. Later on, Gromov have extended the proof to more general Riemannian manifold with Ricci curvature bounded form below. The L\'evy--Gromov comparison theorem can be then stated as follows.   

\begin{theorem} [L\'evy--Gromov]
Let $(M,g, \mu)$ be a compact Riemannian manifold of dimension $n \geq 2$ with Ricci curvature bounded from below by $\kappa > 0$ (in the sense that $Ric_{(M,g)} \geq \kappa g$), with $\mu = \frac{vol}{vol(M)}$ its normalized measure. 

Then its isoperimetric profile satisfies $\mathcal{I}_{(M,g, \mu)} \geq \mathcal{I}_{(\mathbb{S}_{\kappa}^n, g_{\kappa}, \sigma_{\kappa,n})}$, where $\mathbb{S}_{\kappa}^n$ denotes the $n$ dimensional sphere whose Ricci curvature is equal to $\kappa$ - that is of radius $\sqrt{\frac{n-1}{\kappa}}$ - equipped with its normalized surface measure $ \sigma_{\kappa,n}$. 
\end{theorem}

The above statement links curvature and dimension. Curvature and dimension are also linked by Lichnerowicz' minoration of the first eigenvalue of the Laplace Beltrami operator $\lambda_1$. In the above setting, Lichnerowicz theorem express that $\lambda_1 \geq \frac{n \kappa}{n-1}$ . Actually, this constant arises in classical functional inequalities that will be recalled below. First of all, Lichnerowicz' minoration is equivalent to the following spectral gap inequality (by classical integration by parts):
\[
\forall f \in H^2(M), \, \mathrm{Var}_{\mu} (f) = \int_M f^2 d\mu - \bigg( \int_M f d\mu \bigg)^2  \leq \frac{n-1}{n \kappa} \int_{M} | \nabla f |^2 d\mu,
\]
where we have denoted $H^{2}(M) = \{f : \, M \to \R,  \, \int_{M} f^2 d\mu  + \int_{M} | \nabla f |^2 d\mu <  + \infty \}.$

A further important functional inequality is given by the following Sobolev logarithmic inequality. For all $f \in L^2 \log L^2,$
\[
\mathrm{Ent}_{\mu} (f^2) =  \int_M f^2 \log f^2 d\mu - \bigg( \int_M f^2  d\mu \bigg) \log \bigg( \int_M f^2 d\mu \bigg)  \leq 2 \frac{n-1}{n \kappa} \int_{M} | \nabla f |^2 d\mu.
\]

The optimal constant $\frac{n \kappa}{n-1}$ in this inequality is somewhat challenging to reach. For instance, in the case of the Spheres $\mathbb{S}_{\kappa}^n$, $n \geq 2$, it has been established by Mueller and Weissler \cite{M-W} only few years before \cite{B-E}. 

For the further purposes, as a consequence of an inequality linking spectral gap constant and Sobolev logarithmic constant originally proven by Rothaus \cite{Rot}, the Obata rigidity theorem holds for the Sobolev logarithmic inequality (see also \cite{B-G-L}). That is any $n$ dimensional Riemannian manifold with Ricci curvature bounded from below by $\kappa$ such that its Sobolev logarithmic constant is equal to $\frac{n\kappa}{n-1}$ is isometric to $\mathbb{S}_{\kappa}^n$.

\vspace{2mm}

It is an observation actually going back to Poincar\'e that the Gaussian space $(\R^n, \gamma_n)$ where $\gamma_n(dx) = \frac{1}{\sqrt{2 \pi}} e^{- |x|^2/2} dx$ designs the standard Gaussian measure can be seen as the limit in $N$ of $N$ dimensional Euclidean spheres of radius $\sqrt{N-1}$. More presicely, the uniform measure $\mu_N$ on these spheres projected on a fixed subsace $\R^n$ converges to the measure $\gamma_n$. This observation combined with the L\'evy--Gromov's theorem is at the root of the Gaussian isoperimetry, proved in the mid 70's independently by Borell and Sudiakov--Tsirelson. The isoperimetric sets are the half spaces. That is, if $A \subset \R^n$ and $H$ is an half space such that $\gamma_n(A) = \gamma_n(H)$, then
\[
\gamma^{+}(A) \geq \gamma^{+}(H) = \varphi (\Phi^{-1}(\gamma(A))).
\]
Here, and throughout this paper, $\varphi$ designs the density of $\gamma_1$ and $\Phi (x) = \int_{-\infty}^x \varphi(t)dt$ its cummulative distribution function. 

It is an immediate consequence that the isoperimetric profile $\mathcal{I}_{\gamma}$ of the Gaussian space is independant of $n$ and given by the function $ \varphi \circ \Phi^{-1}$. Around twenty years later, a proof of Bobkov of a functional version led to an extension of L\'evy--Gromov comparison theorem by Bakry--Ledoux, in a fairly abstract setting that rely on a \emph{curvature dimension} condition denoted by $CD(\kappa, \infty)$, $\kappa \in \R,$ and that will be described in Section $2$. In particular, for Riemannian manifolds, the result is the following analogous of L\'evy--Gromov comparison theorem. 

\begin{theorem} [Bakry--Ledoux]
Let $(M, g, e^{-\psi} d\mathrm{vol})$ be a weighted Riemannian manifold such that $\mathrm{Ric}_g + \nabla^2 \psi \geq \kappa g$ (unformly as symmetric tensors), with $\kappa >0$.  Then its isoperimetric profile satisfies 
\[
\mathcal{I}_{\mu} \geq \sqrt{\kappa} \mathcal{I}_{\gamma} = \mathcal{I}_{\gamma_{\kappa}}
\]
where $\gamma_{\kappa}$ is the Gaussian measure on $\R$ of variance $\kappa^{-1}$.
\end{theorem}

In addition of the proof of Bakry--Ledoux, two others proofs of this theorem are known. One by Bobkov for log concave measures on $\R^n$ using localization techniques and one by Franck Morgan (\cite{Mo}) for manifold with densities using more geometric arguments. Notice also that the Bakry--Ledoux setting is the more general, as it allows for more general structure than Riemmanian manifolds.

Concerning functionnal inequalities, in the setting of the above theorem, sharp spectral gap and Sobolev inequalities has been established by Bakry and \'Emery (see \cite{B-E}). It holds 
\[
\forall f \in H^2(M), \, \mathrm{Var}_{\mu} (f) \leq \frac{1}{\kappa} \int_{M} | \nabla f |^2 d\mu,
\]
and 
\[
\forall f \in H^2(M), \, \mathrm{Ent}_{\mu} (f^2) \leq \frac{1}{\kappa} \int_{M} | \nabla f |^2 d\mu.
\]
Moreover, the constant $\kappa$ is sharp and attained for the Gaussian space $(\R, | \, \cdot \, |, \gamma_{\kappa})$.

Towards rigidity, for the isoperimetric comparison a statement follows from the proof of Franck Morgan. The result is the following : if there exists $v \in (0,1)$ such that $\mathcal{I}_{\mu} (v) = \mathcal{I}_{\gamma} (v)$, then the manifold splits as $(M',g',\mu') \times (\R, | \, \cdot \, |, \gamma_k)$ where $(M',g',\mu')$ is such that $\mathrm{Ric}_g' + \nabla^2 \psi' \geq \kappa g'$ if $\mu' = e^{-\psi'} d\mathrm{vol}'$. Such rigidity can be seen as an infinite dimensionnal analogous statement of the results of Bayle and Milman. It is a natural guess that the same splitting phenomenon holds for the functionnal inequalities, i.e. if the manifold $(M, g, e^{-\psi} d\mathrm{vol})$ reach the sharp $\kappa$ spectral or Sobolev constant.

\subsection{Statements of our results}

In this paper, we establish some rigidity statements in the infinite dimension setting (i.e. spaces satisfying appropriate curvature dimension conditions - see next section for a precise definition). In particular, following the result of Klartag \cite{Kl} about needles decomposition and the ideas of \cite{C-M2}, we establish an analogous statement of Obata's theorem for manifolds with densities. 

\begin{theorem} \label{infObata}
Let $(M, g, e^{-\psi} d\mathrm{vol})$ a Riemannian manifold with the following condition $\mathrm{Ric}_g + \nabla^2 \psi \geq \kappa g$.  Assume that $\kappa$ is either the spectral gap constant or the Sobolev logarithmic constant. Then there exists another manifold $(M', g', e^{-\psi '}d\mathrm{vol'})$ with the same condition  $\mathrm{Ric}_{g'} + \nabla^2 \psi' \geq \kappa$ such that $M$ splits as
\[
(M, g, e^{-\psi} d\mathrm{vol}) = (\R, | \cdot |, d\gamma_{\kappa}) \times (M', g',e^{-\psi '} d\mathrm{vol'}),
\]
where $d\gamma_{\kappa} = e^{- \kappa x^2 / 2 } dx$ is the Gaussian measure on the real line with variance $1/ \kappa$.
\end{theorem}
Moreover the eigenfunctions associated to the constant $\kappa$ are linear functions in the Gaussian direction. 
The same conclusion holds for the optimal Sobolev logarithmic constant and in the same manner, any non constant function achieving equality in the Sobolev logarithmic inequality is of the same form.

In particular, the last assertion is an extension of Carlen results about characterization of the equality cases in Sobolev logarithmic inequality. 

Concerning the isoperimetric problem, Carlen and Kerce \cite{C-K} have used the Bakry and Ledoux semigroup proof to characterize the isoperimetric minimizers of the Gaussian space without condition of smoothness. Extending the ideas of \cite{C-K}, we re derive the rigidity result of Franck Morgan.  

\begin{theorem} \label{frmor}
Let $(M, d, \mu)$ a weighted Riemmannian manifold of $CD(\kappa, \infty)$ class. If there exists a $v \in (0,1)$ such that $\mathcal{I}_{(M, \mu)} (v) = \mathcal{I}_{(\R, \gamma_{\kappa})} (v)$ then $(M, g, \mu) = (\R, | \, \cdot \, |, \gamma_{\kappa}) \times (M', g',\mu')$. Moreover the isoperimetric minimizers are of the form $( - \infty, \Phi_{\kappa}^{-1}(v)) \times M'$.
\end{theorem}

The main property of the standard (i.e. with $\kappa =1$) Gaussian isoperimetric profile is its independence with respect to the dimension. Above this rigidity theorem, a more difficult question is to address the cases of near equality while keeping dimension free bounds. In order to investigate case of near equality in isoperimetric inequality, the semigroup proof of Bakry and Ledoux seems well-suited as it describe a monotone evolution along the heat-flow. This idea has been put in shape by Mossel and Neeman \cite{M-N1}, who where the first to derive dimension free bounds for the Gaussian isoperimetry deficit $\delta$, with a dependance of order $\log^{-1/2} ( \delta^{-1})$. Notice that the optimal dependance in $\delta$ for the Gaussian space is $\sqrt{\delta}$ and has been established using different techniques since then in \cite{B-B-J} (see also \cite{Eld}). Within the scheme of proof of \cite{M-N1}, it does not seem to be possible to even reduce the dependance to a power of $\delta$. However there is - we believe - two main advantages of \cite{M-N1} with respect to \cite{Eld} and \cite{B-B-J}. First, the proof is somewhat simpler and more importantly, most of the arguments of \cite{M-N1} are valid in \textit{general} $CD(1, \infty)$ spaces. In the last section, we extend the proof of Mossen and Neeamn to obtain a quantitative and dimension free estimates for the Bobkov's inequality in the case of high dimensional Euclidean spheres of radius $\sqrt{n-1}$. Our main result can be stated as follows.

\begin{theorem} \label{thmdefsph}
Let $(\mathbb{S}_1^n,\mu)$ be the Euclidean sphere of radius $\sqrt{n-1}$ endowed with the uniform probability measure $\mu$. Let $A$ be a (Borel measurable) subset of $\mathbb{S}_1^n$ such that
\[
\mu^+(A) \leq \mathcal{I}_{\gamma} (\mu(A)) + \delta.
\]
Then, denoting $\delta_n = \max(\delta, 1/n)$, there exists a spherical cap $H$  and a positive constant $c \in (0.49, 1/2)$ such that  
\[
\mu(A \Delta H) \leq O( |\log \delta_n|^{-c}).
\]
\end{theorem}

As we shall show in the last section, this theorem implies some results about the spherical deficit isoperimetry himself. Since the scheme of proof is of abstract flavor, we will also discuss the case of log-concave probability measures. 

\medskip

This paper is organized as follows. In the next section we briefly recall the general setting in which we will work. In Section~\ref{sec:infdim} we make some comments about the infinite dimensional setting. Then we recall the functional and geometric inequalities. In Section~\ref{sec:stein} we prove Theorem \ref{infObata}. In Section~\ref{sec:rigi} we show how to deduce a rigidity statement for the isoperimetric problem and Theorem \ref{frmor}. Then the last two sections are devoted to the deficit problem in the Bakry--Ledoux comparison Theorem and the proof of Theorem \ref{thmdefsph}.

\section{General Markov framework}

This section aims at presenting the abstract framework that is subsequent to the synthetic notion due to Bakry--\'Emery \cite{B-E} of Curvature--Dimension condition. For a comprehensive introduction to the topic, we refer the reader to \cite{B-G-L}, or \cite{Le1} for a shorter - but nonetheless complete - overview. Recall that curvature and dimension are linked in a Riemmanian setting by Lichnerowicz's minoration of the first eigenvalue of the Laplace operator. In a complete Riemannian manifold, the Bochner's formula expresses that
\[
\frac{1}{2} \Delta (|\nabla f|^2) - \nabla \Delta f \cdot \nabla f = \mathrm{Ric}_g (\nabla f, \nabla f) + \| \nabla^2(f) \|_{HS}^2.
\]
If the manifold is $n$ dimensional, Schwarz's inequality ensures that $ \|\nabla^2(f) \|_{HS}^2 \geq \frac{1}{n} (\Delta f)^2$. It implies by a standard integration by parts Lichnerowicz' minoration of the first eigenvalue of the Laplace Beltrami operator. The Bochner's formula is a key observation that led the authors of \cite{B-E} to define a generalized curvature-dimension criterion $CD(\kappa, n)$, $\kappa \geq 0, n \in [0, \infty]$ for a large class of spaces. This criterion has proven to be very useful to reach sharp geometric and functional inequalities. We briefly recall a few basic definitions. 

\vskip 2mm

Let $(E, \mu)$ be a probability space equipped with $L$ a diffusion operator acting on a domain $\mathcal{A}$, that is such that for each $\psi \, : \R^k \to \R$ with $\psi(0) =0$, and $f_1, \ldots, f_k, \in \mathcal{A},$
\[
L \psi(f_1, \ldots, f_k) = \sum_{i=1}^k \partial_i \psi (f_1, \ldots, f_k) L f_i +  \sum_{1 \leq i,j \leq k} \partial_{ij}^2 \psi (f_1, \ldots, f_k) \Gamma(f_i, f_j).   
\]

A diffusion $L$ is generating a semigroup - that is a family of operators $(P_t)_{t \geq 0}$ such that $P_0 = Id$ and $P_{t+s} = P_t \circ P_s$ - by taking $(P_t   := e^{tL}))_{t \geq 0}$. These semigroups $(P_t)_{t \geq 0}$ are Markov that is 
\[
\forall t \geq 0, \, P_t {\bf 1} = {\bf 1},
\]
where $\bf 1$ is the constant function defined on $E$ equal to $1$. We shall assume moreover that the measures $\mu$ are reversible with respect to $L$ or $(P_t)_{t \geq 0}$ that is
\[
\forall f, g \in L^2(\mu), \, \int_{E} f Lg \, d\mu = \int_{E} g Lf \, d\mu,
\]
and invariants with respect to $(P_t)_{t \geq 0}$ that is
\[
\forall f \in L^1(\mu), \, \int_{E} P_t f  \, d\mu = \int_{E} f \, d\mu.
\]
Another important property is the ergodicity of these semigroups that is 
\[
\lim_{s \to \infty} P_s f  = \int_E f d\mu,
\] 
for every function $f$ (say integrable with respect to $\mu$).

The Laplace Beltrami operator over a Riemmanian manifold is a relevant example of diffusion, which generates the heat semigroup $(e^{t \Delta})_{t \geq 0}$. Following Paul Andr\'e Meyer, Bakry and Emery define the ``carr\'e du champ'' operator $\Gamma$ ($\Gamma_1$ below) and its iterations $(\Gamma_n)_{n \geq 0}$ by setting $\Gamma_0(f,g) = fg$ and then by induction
\[
\forall n \geq 1, \quad \Gamma_n(f,g) = \frac{1}{2} \bigg( L(\Gamma_{n-1}(f,g)) - \Gamma_{n-1}(f, Lg) -  \Gamma_{n-1}(g, Lf) \bigg).
\] 
The diffusion property implies the following chain rule formulas
\[
\Gamma(u(f)) = u'(f)^2 \Gamma (f), \qquad \Gamma_2(u(f)) = u'(f)^2 \Gamma_2(f) + u'(f) u''(f) \Gamma(f, \Gamma(f)) + u''(f)^2 \Gamma(f)^2.
\]
When $(M, g,  \mathrm{vol})$ is a Riemannian manifold and the diffusion is the Laplace Beltrami operator $\Delta$, the carr\'e du champ $\Gamma(f)$ is equal to $| \nabla f |^2$, and Bochner's formula can then be written as 
\[
\Gamma_2 (f) =  \mathrm{Ric}_g (\nabla f, \nabla f) + \| \nabla^2 f \|_{HS}^2.
\]
As a consequence, if the manifold is $n$ dimensional with Ricci curvature bounded below by $\kappa$, the Bochner's formula implies 
\[
\Gamma_2 (f) \geq \kappa \Gamma (f) + \frac{1}{n} (L f)^2, 
\]
which is a good formulation for a general \emph{curvature-dimension} definition. 
\begin{definition}
Say that $(E, \mu, \Gamma)$ satisfy the curvature-dimension criterion $CD(\kappa, n)$ if 
\[
\forall f \in \mathcal{A}, \, \Gamma_2 (f) \geq \kappa \Gamma (f) + \frac{1}{n} (L f)^2. 
\]
\end{definition}

Obviously, it is clear that $CD(\kappa, n) \Rightarrow CD(\kappa', n)$ for $\kappa' \leq \kappa$ and $CD(\kappa, n) \Rightarrow CD(\kappa, m)$ for $n \leq m$. The curvature-dimension condition $CD(\kappa,n)$ is therefore a condition of curvature bounded from below  by $\kappa$ and dimension bounded from above by $n$.

If $n \in \N$ and $\kappa >0$, the model spaces for this condition are naturally given by the Euclidean Spheres $\mathbb{S}^n$ of radius $\sqrt{\frac{n-1}{\kappa}}$,  
and moreover Bochner's formula indicates that $\Gamma_2 (f) = \kappa \Gamma (f) + \| \nabla^2(f) \|_{HS}^2$. 

More generally one can take weighted Riemannian manifold, which is Riemannian manifold with density $d\mu = e^{-V} d\mathrm{vol}$ with the associated diffusion operator $L = \Delta - \nabla V \cdot \nabla$, called the Witten-Laplacian. Then the ``carr\'e du champ'' $\Gamma(f)$ is still given by $|\nabla f|^2$ and by the Bochner's formula, the $CD(\kappa, \infty)$ condition $ \Gamma_2 (f) \geq \kappa \Gamma (f)$ turns out to be true whenever $\nabla^2 V + \mathrm{Ric}_g \geq \kappa$ (uniformly as symmetric tensors). This tensor is called \textit{Bakry-Emery tensor}. 

The Gaussian space $(\R^n, \gamma_n)$ , $n\geq 1$, is an example of a space with infinite dimension and curvature $1$ (as immediately checked). Therefore, in this definition, the dimension do not necessary coincide with the topological dimension. A good explanation of this fact is given by the Poincar\'e observation. Recall indeed that the Gaussian space can be seen as limit of $n$-dimensional spheres of radius $\sqrt{n-1}$, that is of objects of class $CD(1,n)$. This is also a good explanation of the fact that the geometric and functional inequalities in Gaussian space are \textit{dimension free}.  

The underlying semigroup in Gaussian space $(Q_t)_{t \geq 0}$ is the Ornstein-Uhlenbeck semigroup. It acts on functions $f \, : \, \R^n \to \R$ of the Dirichlet domain, by
\[
x \in \R^n \mapsto  Q_t f(x) = \int_{\R^n} f(e^{-t} x + \sqrt{1 - e^{-2t}}y) d\gamma (y) =  \int_{\R^n} f(y) q_t(x,y) d\gamma (y),
\]
with  $q_t$ the Mehler kernel associated to the Ornstein-Uhlenbeck semigroup.

Since $\R^n$ is flat, $(\R^n, \mu(dx) =e^{-V(x)}dx)$ satisfies $CD(\kappa, \infty)$ whenever $\nabla^2 V \geq \kappa$, which corresponds to the class of strictly log-concave measures. Indeed, the iterated ``carr\'e du champ'' $\Gamma_2$ takes the following form : 
\[
\Gamma_2 (f) = \nabla^2 V (\nabla f, \nabla f) + \| \nabla^2 f \|_{HS}^2.
\]

Of course, this curvature dimension criterion in its abstract formulation covers a wider setting than Riemannian manifolds. One still need a diffusion operator, hence some smooth structure.

The principal strength of this abstract formulation is that it has been proven to be a very efficient criterion to established sharp geometric and functional inequalities. In this context, say that $(E, \mu, \Gamma)$ satisfies $SG (\lambda)$ (spectral gap, or Poincar\'e, inequality with constant $\lambda$) if $\lambda$ is the best (i.e. largest) positive constant such that for all $f \in H^2(M)$,
\[
\mathrm{Var}_{\mu} (f) \leq \frac{1}{\lambda} \int_{M} | \nabla f |^2 d\mu.
\]
In the same manner, say that $(E, \mu, \Gamma)$ satisfies $LS (\rho)$ (Sobolev logarithmic inequality with constant $\rho$) if $\rho$ is the best (i.e. largest) positive constant such that for all $f \in H^2(M)$,
\[
\mathrm{Ent}_{\mu} (f^2)   \leq \frac{2}{\rho} \int_{M} | \nabla f |^2 d\mu.
\]

In \cite{B-E}, the authors show that the $CD(\kappa, n)$-condition implies sharp Sobolev logarithmic inequalities, i.e with $\rho = \frac{n \kappa}{n-1}$ ($= \kappa$ if $n = \infty$), equivalent to the important property of hypercontractivity of the underlying semigroups. We confer to reader to \cite{Ba}, \cite{B-G-L} for the analogous dimensional Sobolev inequalities in the case $n < \infty$. 

In the case $n = \infty$, referred as infinite dimensional setting, a common scheme of proof to established these inequalities is to interpolate along the corresponding semigroups and to show monotony along these semigroups. We will illustrate the later in Section~\ref{sec:geofun}. It is therefore well suited to investigate cases of equality. Before doing so, in the next section, we discuss the analogous rigidity statements in the case of weighted Riemannian manifolds. 

\section{The infinite dimension setting} \label{sec:infdim}

\vspace{2mm}

We now focus on the infinite dimensional setting corresponding to the $CD(\kappa, \infty)$ condition. Recall that in this context, when $\kappa >0$, the model space is the Gaussian space $(\R^n, \gamma_{n,\kappa})$, where $\gamma_{n,\kappa}  = \gamma_{\kappa}^{\otimes n}$, $n \geq 1$ and that it satisfies $LS(\kappa)$ and $SG(\kappa)$ for all $n \geq 1$. Theses important inequalities has been established well before \cite{B-E}, in the early XXth century for the spectral gap inequality and in the 70's for Sobolev logarithmic inequalities.

An analogous of Obata's theorem is the following statement : if a weighted manifold $(M, g, \mu)$ have the optimal constant $\kappa$ for the spectral gap inequality, then the manifold splits as $(M, g, \mu) = (\R, | \, \, |, \gamma_{\kappa}) \times (M', g', \mu')$. As $\rho \geq \theta \kappa + (1 - \theta) \lambda$ for some $\theta \in (0, 1)$ the same conclusion would hold for the Sobolev logarithmic inequality. We will prove this conclusion whenever there is non constant extremal functions for this inequalities. Since equality for the spectral gap inequality is attain for eigenfunctions of diffusion operator $L = \Delta - \nabla \psi \cdot \nabla$ associated to the eignenvalue $\kappa$, this implies the statement.  
Furthermore, we can establish characterization of these extremal functions. We note that this conclusion overlaps with the work of Bentaleb \cite{Ben}, who uses very similar tools as ours (i.e. interpolations along heat flows). For the reader's convenience, we will still give the complete proofs in the next sections. 

\vspace{2mm}

Concerning the isoperimetric problem, a rigidity theorem for the Bakry--Ledoux's comparison theorem has been established by Franck Morgan for manifolds with densities. The proof relies on deep results from geometric measure theory which ensure that in such setting there is always a regular set that minimize the isoperimetric problem. Note that for the real line, Bobkov \cite{Bo1} has proved that minimizers exist and are half-lines in the weaker class of log concave measures (not strictly and without regularity assumptions on the potential). We will make the more general observation that if there is a non trivial minimizer in an abstract Markov triple, then some kind of rigidity occurs. In particular, we recover Morgan's theorem for isoperimetric problem with characterization of the optimal sets in the setting of a weighted Riemannian Manifold. 

The starting point is a functional form of the Gaussian isoperimetry. In a remarkable work \cite{Bo2}, Bobkov showed that the Gaussian isoperimetry is equivalent to the following functional inequality
\begin{equation} \label{eq.Bobkov} 
\mathcal{I}_{\gamma} \bigg(\int_{\R^n} f d\gamma_n \bigg) \leq 
\int_{\R^n} \sqrt{ \mathcal{I}_{\gamma}^2 (f) + |\nabla f |^2 } d\gamma_n
\end{equation}
for each function $f$ : $\R^n \to [0,1]$ locally Lipschitz. Using this inequality to a sequence of Lipschitz functions $(f_{\varepsilon})_{\varepsilon \geq 0}$ approaching $\mathbf{1}_A$ as $\varepsilon$ goes to $0$ implies back $I_{\gamma} (\gamma(A)) \leq \gamma^{+}(A)$.
Bobkov's original proof is build in three steps. Firstly, one proves the inequality on the two point space $\{0,1\}$, the using tensorization property of the inequality one establishes it on the discrete cubes $\{0,1\}^n$, $n \geq 1$ before proving it on the Gaussian space by means of a central limit argument.

Noticing that by ergodicity, $\int_{\R^n} f d\gamma_n  = \lim_{t \to \infty} Q_t f$, and that $f = Q_0 f$, Bakry and Ledoux's idea \cite{B-L1} is to consider the function 
\[
\psi \, :  s \in [0, \infty ) \mapsto \int_{\R^n} \sqrt{ \mathcal{I}_{\gamma}^2 (Q_s f) + |\nabla Q_s f |^2 } d\gamma_n
\]
and showing that it is non increasing, recovering therefore Bobkov inequality. 

A main advantage of this proof and its tools is that it can be easily adapted on a Markov triple $(E, \mu, \Gamma)$ with the condition $CD(\kappa, \infty)$ described in the previous section. The result is the following 
\begin{equation} \label{eq.BobBL} 
\mathcal{I}_{\gamma} \bigg(\int_E f d\mu \bigg) \leq 
\int_E \sqrt{\mathcal{I}_{\gamma}^2 (f) + \frac{1}{\kappa^2} \Gamma (f)} d\mu
\end{equation}
for each function $f$ : $E \to [0,1]$ locally Lipschitz (in an appropriate sense). Since for weighted Riemannian manifolds, $\Gamma(f) = | \nabla f |^2$ and $\mu^+(A) = \lim_{k \to \infty} \|\nabla f_k \|_{L^1(\mu)}$ for a sequence of functions $(f_k)_{k \geq 0}$ approaching $\bf{1}_A$, it yields the comparison theorem announced in the introduction.

To close the picture over a Riemmanian manifold $(M, g, \mu = e^{-V} d\mathrm{vol})$, Barthe--Maurey \cite{B-M} have shown the equivalence between a comparison theorem over isoperimetric profiles 
\[
\forall A \subset M, \, \mu^{+}(A) \geq \kappa \mathcal{I}_{\gamma} (\mu(A)) 
\]
and the Bobkov inequality \eqref{eq.BobBL} for all Lipschitz functions mapping to $[0,1]$. For the further purposes, let us discuss the case of the Euclidean Spheres $\mathbb{S}^n$. Standard arguments (see e.g. \cite{B-M}) imply that 
\[
\inf_{a \in (0,1)} \frac{\mathcal{I}_{(\mathbb{S}^n,g,\mu) }(a)}{\mathcal{I}_{\gamma}( a)} = \frac{\mathcal{I}_{(\mathbb{S}^n,g, \mu)} (1/2)}{\mathcal{I}_{\gamma}(1/2)} =  \sqrt{2\pi} \frac{\mathrm{vol}_{n-1}(\mathbb{S}_{n-1})}{\mathrm{vol}_n(\mathbb{S}_n)} = \sqrt{2} \frac{\Gamma(\frac{n+1}{2})}{\Gamma(\frac{n}{2})} = c_n,
\]
and the latter is always larger than $\sqrt{n-1}$, and therefore always less than $\sqrt{n}$. This can be seen as a consequence of log-convexity of the Gamma function.  \footnote{Indeed by log-convexity of the Gamma function, $\frac{n-1}{2}\Gamma(\frac{n}{2})^2 \leq \frac{n-1}{2}\Gamma(\frac{n-1}{2}) \Gamma(\frac{n+1}{2}) =\Gamma(\frac{n+1}{2})^2$, which is equivalent to the claim.}
The optimal Bobkov inequality can therefore be written as
\[
\mathcal{I}_{\gamma} \bigg(\int_{\mathbb{S}^n} f d\mu \bigg) \leq 
\int_{\mathbb{S}^n} \sqrt{ \mathcal{I}_{\gamma}^2 (f) + \frac{1}{c_n^2}|\nabla f |^2 } d\mu.
\]

As we will see below, the isoperimetric problem admits a minimizer for every manifold with density. Towards rigidity, whenever $\mathcal{I}_{(M,g, \mu)} = \mathcal{I}_{\gamma_{\kappa}}$, this will imply that there is always a non constant extremal function for the Bobkov's inequality. This in turn implies the splitting theorem of Franck Morgan.

\section{Geometric and functional inequalities} \label{sec:geofun}

In this section, we recall how to reach sharp geometric and functional inequalities in the abstract framework of $CD(\kappa, \infty)$ spaces described in Section $2$. The starting point is that the curvature dimension criterion of Bakry--Emery implies a commutation between $\Gamma$ and the semigroup $(P_t)_{t \geq 0}$ in the following sense.  

\begin{lemma}
Under the $CD(\kappa, \infty)$ condition, $\kappa \in \R,$ we have for all $t \geq 0$,
\begin{equation}  \label{eq.comm}
\Gamma(P_t f) \leq e^{-2\kappa t} P_t (\Gamma(f)).
\end{equation}
\end{lemma}

The proof of this lemma is very simple. 
\begin{proof}
Define $\psi_t \, : \, s \mapsto e^{-2\kappa s} P_s (\Gamma(P_{t-s} f))$. Then using the semigroup property and the fact that $\Gamma$ is bilinear, 
\[
\psi'_t(s) = e^{-2 \kappa s} [P_s (2\Gamma_2 (P_{t-s} f) - 2 \kappa P_s (\Gamma(P_{t-s} f))] = 2e^{-2 \kappa s} [P_s (\Gamma_2 (P_{t-s} f) - \kappa \Gamma(P_{t-s} f))] \geq 0,
\] 
by the $CD(\kappa, \infty)$ condition. But the statement of the lemma reads as $\psi_t(0) \leq \psi_t(t)$.
\end{proof}
For the further purposes, the proof implies that there is equality in \eqref{eq.comm} if and only if for all $s \in (0,t)$, $\Gamma_2 (P_{t-s} f) = \kappa \Gamma (P_{t-s} f)$ so that taking the limit $s \to t$, $\Gamma_2 (f) = \kappa \Gamma (f)$.

It is relevant to point out that in the case of the heat semigroup $(P_t)_{t \geq 0}$ on the Euclidean spaces and the Ornstein--Uhlenbeck semigroup $(Q_t)_{t \geq 0}$ on the Gaussian space we have an exact commutation, which follows immediately from the explicit formulas of such semigroups. That is, for smooth functions $f$ of the respective Dirichlet domains, $\nabla P_t f= P_t \nabla f$ and $\nabla Q_t f = e^{-t} Q_t \nabla f$.

\subsection{Spectral gap and Sobolev logarithmic inequalities : direct and reverse forms.}

Using the above lemma, we reach the following propositions (see e.g. \cite{Le1}, \cite{B-G-L} - actually for the second one, the \textit{a priori} stronger, but in fact equivalent, property $\sqrt{\Gamma(P_t f)} \leq e^{-\kappa t} P_t (\sqrt{\Gamma(f)})$ is needed).  

\begin{proposition}
Let $(E, \mu, \Gamma)$ be a probability space satisfying $CD(\kappa, \infty)$.  Then,
\[
\forall f \in \mathcal{A}, \quad C(\kappa, t) \Gamma(P_t f)   \leq P_t (f^2) - (P_t f)^2  \leq  D(\kappa, t) P_t (\Gamma(f)), 
\]  
where $C(\kappa, t) = 2\int_0^t e^{2\kappa s} ds$ and $D(\kappa, t) = 2\int_0^t e^{-2\kappa u} du$.  In particular, if $\kappa >0$, $(E, \mu, \Gamma)$ satisfies $SG(\kappa)$.
\end{proposition}

\begin{proposition}
Let $(E, \mu, \Gamma)$ be a probability space satisfying $CD(\kappa, \infty)$. Then, 
\[
\forall f \in \mathcal{A}, \quad C(\kappa, t) \frac{\Gamma(P_t f)}{P_t f}   \leq P_t (f \log f) - P_t f \log P_t f  \leq  D(\kappa, t)P_t \bigg( \frac{\Gamma(f)}{f} \bigg), 
\]  
where again $C(\kappa, t) = 2\int_0^t e^{2\kappa s} ds $ and $D(\kappa, t) = 2\int_0^t e^{-2\kappa s} ds.$  In particular, if $\kappa >0$, $(E, \mu, \Gamma)$ satisfies $LS(\kappa)$.
\end{proposition}

\medskip

Let us discuss the reverse form of these inequalities. The reverse Poincar\'e inequality implies that each bounded (by one) function is $1/\sqrt{C(\kappa,t)}$-Lipschitz in the sense that $\Gamma(P_t f) \leq 1/C(\kappa, t)$. It is also known that it implies the following lemma (see \cite{Le1}).

\begin{lemma} Let $(E, \mu)$ satisfying the $CD(0, \infty)$ condition. Then
\begin{equation} \label{reversesg}
\forall f \in L^1(E), \, \| f - P_t f \|_1 \leq  \sqrt{2t} \| \sqrt{\Gamma(f)}\|_1.
\end{equation}
\end{lemma}

Turning to the reverse Sobolev Logarithmic inequality, it implies that whenever $0 \leq f \leq 1$, 
\[
C(\kappa,t) \Gamma(\log P_t f) \leq -\log (P_t f)
\]
so that $\phi = (-\log P_t f)^{1/2}$ is $\frac{1}{2\sqrt{C(\kappa,t)}}$-Lipschitz in the sense that $\Gamma(\phi) \leq \frac{1}{4C(\kappa,t)}$. 

\subsection{Bakry--Ledoux comparison theorem and reverse isoperimetric inequality}

As discussed in Section~\ref{sec:infdim}, given a Markov triple $(E, \mu, \Gamma)$ with the condition $CD(\kappa, \infty)$, $\kappa >0$, Bakry and Ledoux gave a semigroup proof of Bobkov's inequality. Using this proof, Carlen and Kerce \cite{C-K} has characterized cases of equality. While some of the computations arising are somewhat tedious, the following proposition is a consequence of \cite{B-L1}, which implies the Bobkov inequality. 

\begin{proposition}
Let $(E, \mu, \Gamma)$ satisfying the $CD(\kappa, \infty)$ condition and $f$ : $E \mapsto [0,1]$. Set
\[
\psi \, :  s \in (0, \infty) \mapsto  \int_{E}  \sqrt{ \mathcal{I}_{\gamma}^2 (P_s f) + \kappa^{-2} \Gamma (P_s f)} \,  d\mu.
\] 
Then
\[
- \psi'(s) \geq  \int_{E} \frac{\mathcal{I}_{\gamma}(P_s f) (\Gamma_2 - \kappa \Gamma) (\Phi^{-1} \circ P_s f)}{(1 + \Gamma (\Phi^{-1} \circ P_s f))^{3/2}} d\mu.
\]
\end{proposition}

\begin{proof}
The proof relies on the work of Bakry--Ledoux. Denoting  $P_s f = f_s$, it is shown in \cite{B-L1} that
\[
- \psi'(s) \geq \int_{E} \frac{\mathcal{I}^2_{\gamma}(f_s) (\Gamma_2 - \kappa \Gamma )(f_s)   -\mathcal{I}_{\gamma}(f_s) \mathcal{I}'_{\gamma}(f_s) \Gamma (f_s,\Gamma (f_s)) + \mathcal{I}'_{\gamma}(f_s)^2 \Gamma(f_s)^2}{(\mathcal{I}_{\gamma}^2 (f_s) + \Gamma (f_s)^{3/2})} d\mu.
\] 
Following the idea of Carlen and Kerce, we express the derivative $\psi'$ in terms of $\Phi^{-1} \circ f_s$. There is actually some geometric intuition behind this change of variable taking roots in the Bobkov inequality himself. Using the chain rule formulas 
\[
\Gamma(u(f)) = u'(f)^2 \Gamma (f), \qquad \Gamma_2(u(f)) = u'(f)^2 \Gamma_2(f) + u'(f) u''(f) \Gamma(f, \Gamma(f)) + u''(f)^2 \Gamma(f)^2,
\]
with $u = \Phi^{-1}$, since $(\Phi^{-1})' = \frac{1}{\mathcal{I}_{\gamma}}$ and $(\Phi^{-1})'' = - \frac{\mathcal{I}'_{\gamma}}{\mathcal{I}^2_{\gamma}} $ we have
\[
(\mathcal{I}_{\gamma}^2 (f_s) + \Gamma (f_s))^{3/2} = \mathcal{I}_{\gamma}^3 (f_s)(1 + \Gamma (\Phi^{-1} \circ f_s))^{3/2}
\]
and
\[
\mathcal{I}^2_{\gamma}(f_s) (\Gamma_2 - \kappa \Gamma) (f_s) - \mathcal{I}_{\gamma}(f_s) \mathcal{I}'_{\gamma}(f_s) \Gamma (f_s,\Gamma (f_s)) + \mathcal{I}'_{\gamma}(f_s)^2 \Gamma(f_s)^2 
= \mathcal{I}^4_{\gamma}(f_s) (\Gamma_2  - \kappa \Gamma) (\Phi^{-1} \circ f_s), 
\]
which concludes the proposition.  
\end{proof}

Of course, since its derivative is non positive, $\psi$ is decreasing which implies by ergodicity of the semigroup $(P_s)_{s\geq0}$ the Bobkov's inequality, and thus the isoperimetric conclusion.

For the further purposes, it has been established in \cite{B-L1} by the same scheme of proof a reverse isoperimetric inequality in the following form. 
\begin{proposition} \label{reverseiso}
Let $(E, \mu, \Gamma)$ be a probability space satisfying $CD(\kappa, \infty)$. Then, for all functions $f \, : \,  E \to [0,1]$,  
\[
C(\kappa, t) \Gamma(P_t f)  \leq (\mathcal{I}_{\gamma} (P_t f))^2 - (P_t (\mathcal{I}_{\gamma} (f)))^2.
\]  
\end{proposition}
By the chain-rule formula, this inequality expresses that $\Phi^{-1} (P_t f)$ is Lipschitz in the sense that $\Gamma (\Phi^{-1} (P_t f)) \leq \frac{1}{C(\kappa,t)}$.

One can prove that it implies in the Gaussian space a reverse Bobkov inequality (cf \cite{B-L1})). Indeed (say with $\kappa =1)$, using the exact commutation $e^t \nabla Q_t f = Q_t \nabla f$ and letting $t$ goes to infinity, it holds by ergodicity
\[
\bigg| \int_{\R^n} \nabla f d\gamma_n \bigg|^2 \leq \mathcal{I}_{\gamma}^2 \bigg( \int_{\R^n} f d\gamma_n \bigg) - \bigg(\int_{\R^n} \mathcal{I}_{\gamma} (f) d\gamma_n \bigg)^2,
\]
or equivalently, 
\begin{equation} \label{reverseBobkov}
\sqrt{\bigg| \int_{\R^n} \nabla f d\gamma_n \bigg|^2 +  \bigg( \int_{\R^n} \mathcal{I}_{\gamma} (f) d\gamma_n  \bigg)^2} \leq \mathcal{I}_{\gamma} \bigg( \int_{\R^n} f d\gamma_n \bigg).
\end{equation}

\section{Stein's lemma and rigidity theorems} \label{sec:stein}

The aim of this section is to put forward a rigidity phenomenon that occurs when there exists a non constant function achieving equality in the above geometric and functional inequalities. A common factor that appears in the proofs is the non negative quantity $\Gamma_2 - \kappa \Gamma$. In order to have equality, is is necessary that this quantity is equal to $0$.

Without loss of generality, we can assume that in any of the above inequalities, $\int_{E} f d\mu =0$. Let us take a (smooth) function achieving inequality in spectral gap inequality, Sobolev logarithmic or Bobkov inequality. Then, there is equality in the commutation \eqref{eq.comm}, which implies the equality $\Gamma_2 = \kappa \Gamma$. More precisely, $\Gamma_2(f) = \kappa \Gamma(f)$ for the spectral gap inequality, $\Gamma_2(\log f) = \kappa \Gamma (\log f)$ for the Sobolev Logarithmic inequality and $\Gamma_2(\Phi^{-1} \circ f) = \kappa \Gamma (\Phi^{-1} \circ f)$ for the Bobkov's inequality. Besides, we have the following Stein's lemma.

\begin{lemma}
Assume that $(E, \Gamma, \mu)$ satisfies $CD(\kappa, \infty)$ and that there exists $f$ non constant such that $\Gamma_2 (f) = \kappa \Gamma (f)$. Then $L f = - \kappa f$. In particular the law of $f$ is Gaussian. 
\end{lemma}

\begin{proof} [Proof]
Without loss of generality, we can make the assumption that  $\int_E f d\mu =0$. Then by Cauchy-Schwarz' inequality and by the spectral gap inequality, 
\[
\kappa \int_{E} \Gamma (f)  d\mu = \kappa \int_E  f (-Lf)  d\mu  \leq \kappa \bigg( \int_{\Omega}  f^2  d\mu \int_E  (Lf)^2  d\mu   \bigg)^{1/2} \leq \bigg( \kappa \int_E  \Gamma(f)  d\mu \int_E  (Lf)^2  d\mu   \bigg)^{1/2},
\]
so that 
\[
\kappa \int_E \Gamma (f)  d\mu  \leq \int_E  (Lf)^2  d\mu = \int_E  \Gamma_2(f)  d\mu.
\]
Since by hypothesis $\Gamma_2 (f) = \kappa \Gamma (f)$, there is equality everywhere. Then, there exists a negative real number $k$ such that $Lf = kf$ and necessary $k = -\kappa$. 

For the next statement we confer to \cite{Le2}.

\end{proof}

Let $(E, \mu, \Gamma)$ be a $CD(\kappa, n)$ space, with $n < \infty$. It is important to point out that since the first eigenvalue of $L$ is given by $\frac{n \kappa}{n-1}$ there is no function achieving $\Gamma_2 (f) = \kappa \Gamma(f)$. In fact this property is true only for ``real'' $CD( \kappa, \infty)$ spaces, that is not satisfying $CD(\kappa, n)$ for any finite $n$. 

When $(M, g, e^{-\psi} d\mathrm{vol})$ is a weighted Riemannian manifold, as a consequence of the Bochner's formula,
\[
\Gamma_2(f) = \Gamma(f) \Longleftrightarrow \mathrm{Ric}_{g} (\nabla f, \nabla f) + (\nabla^2 \psi - \kappa Id) (\nabla f, \nabla f) + \| \nabla^2 f \|_{\mathrm{HS}}^2 = 0,
\]  
and the conclusion of Stein's lemma reads
\[
\Delta_g f - \nabla \psi \cdot \nabla f = - \kappa f.
\]
When the manifold is $\R^n$, it easily implies that there is a Gaussian marginal and $f$ is an eigenvector of $L$ in the Gaussian direction. In general, we can make use of the needle decompostions of \cite{Kl}, so that it implies the following theorem. 

\begin{theorem} \label{ThmObatab}
Let $ (M, g, e^{-\psi} d\mathrm{vol})$ be a $CD(\kappa, \infty)$ manifold. If there exists a non constant extremizer $f$ in one of the inequality of Section~\ref{sec:geofun}, then
\[
(M, g, e^{-\psi} d\mathrm{vol}) = (\R, | \cdot |, \gamma_{\kappa}) \times (M', g',e^{-\psi '} d\mathrm{vol}),
\]
where $(M', g',e^{-\psi '} d\mathrm{vol})$ is a $CD(\kappa, \infty)$ manifold and $f$, $\log f$ or $\Phi^{-1} (f)$ is an eigenfunction of the Ornstein-Uhlenbeck semigroup associated to the eigenvalue $\kappa$, that is a linear function. 
\end{theorem}

In particular if the spectral gap constant is equal to $\kappa$, an eigenfunction associated to $\kappa$ satisfies the condition of the above theorem. Recall that, if $\rho$ designs the Sobolev logarithmic constant and $\lambda$ the spectral gap constant, then $\rho \geq \theta \kappa + (1 - \theta) \lambda$ with $\theta \in  (0,1)$. Thus, it implies Theorem \ref{infObata} announced in the introduction, which can be viewed as an infinite dimensional analogous of the Obata's theorem. 

In particular, the Gaussian space is the only weighted manifold for which there is non trivial extremizers. The linear functions are the only non constant functions achieving equality in the spectral gap inequality and the exponential functions are the only non constant functions achieving equality in Sobolev logarithmic inequality. The same conclusion holds for the reverse spectral gap and Sobolev logarithmic inequalities. 

For the isoperimetric inequality, we are interested in finding extremal sets. Actually, the fact that half spaces are the only isoperimetric minimizers is a consequence of the last assertion of the proposition. This characterization is due to Carlen and Kerce and the purpose of the next section is the generalize this fact to general $CD(\kappa, \infty)$ weighted Riemannian manifolds. Notice that for the reverse Bobkov inequality \eqref{reverseBobkov} there is equality if and only if $\Phi^{-1} (f)$ is linear. This conclusion has already been obtained in \cite{B-C-F}.

To conclude this section, the condition that there exists a non constant function $f$ such that $\Gamma_2(f) = \kappa\Gamma(f)$  implies $\Delta_g f = 0$, so that $\nabla (\nabla \psi \cdot \nabla f) =  (\nabla^2 \psi) (\nabla f) =  \kappa \nabla f$. Thus, there exists $a \in T_x M$ such that $\nabla f = a$ and moreover $\mathrm{Ric}_{g} (a, a) =0$. One wonder if one can prove that it implies in a more direct way (i.e. wihout relying on needles decomposition) that the manifold is a Cartesian product with a one dimensional Gaussian space. Notice that if $M$ contains a line, a splitting Theorem has been established in \cite{W-W}.

\section{An isoperimetric rigidity statement} \label{sec:rigi}

We now turn to the isoperimetric rigidity statement. Intuitively, $\lim_{t \to 0} = \int_{M} |\nabla P_t \mathbf{1}_A| d\mu$ is a good candidate for the boundary measure of $A$ as  $P_t \mathbf{1}_A$ is a smooth function whenever $t >0$ and $P_t \mathbf{1}_A$ converges to $\mathbf{1}_A$ as $t$ goes to $0$. However we have defined the isoperimetric problem relatively to the Minkownski content $\mu^{+}(A) = \liminf_{\varepsilon \to 0} \frac{\mu(A^{\varepsilon}) - \mu(A)} {\varepsilon}$.  In \cite{C-K}, the isoperimetric problem is given with respect to another definition of boundary, the weighted relative perimeter, that will be recalled below. Still, it is known for a while that in Gaussian space the isoperimetric problem can be taken with respect to any kind of definition of perimeters. The remark in \cite{C-K} is that our naive definition of boundary using the semigroup agrees with the weighted relative perimeter. Actually, as we shall show, this fact is more general.

Let $(M,g,d\mathrm{vol})$ be a Riemmanian manifold. We denote by $|Du|$ the variation, in the sense of De Giorgi \cite{DG}, of $u \in L^1(M)$ defined by  
\[
|Du| (M) =  \sup \bigg \{ \int_{M} u \mathrm{div} \varphi d\mathrm{vol}, \, \varphi \in C_c^1(M, \R)\,  \& \, \| \varphi \|_{\infty} \leq 1. \bigg \}
\]
When $u$ is the indicator function of a set $A$, $|Du|$ is called perimeter of $A$. It has been established in \cite{Car-M} that the variation of $u$ is linked with the standard heat kernel $(P_t)_{t \geq 0} = (e^{t \Delta})_{t \geq 0}$ in the following sense. 
\begin{proposition} \label{boundary} 
For all $u \in L^1(M)$, 
\[
|D u| = \lim_{t \to 0} \int_{M} |\nabla P_t u| d\mathrm{vol}.
\]
\end{proposition}

When $(M,g,\mu)$ is a weighted Riemannian manifold, one naturally extends the definition of De Giorgi and define the weighted-perimeter as follows. 
\begin{definition}
Let $(M, g, \mu = e^{-\psi} \mathrm{vol})$ a weighted Riemannian manifold with $(P_t)_{t \geq 0}$ its relative semigroup. Then we define the relative perimeter of $A \subset M$ by
\[
\mathrm{Per}_{\psi}(A) =  \sup \bigg \{ \int_{A}  (\mathrm{div} \varphi - \nabla \psi \cdot \varphi) d\mu, \, \varphi \in C_c^1(M, \R)\, \& \, \| \varphi \|_{\infty} \leq 1 \bigg \}.
\]
\end{definition}
The proof of \cite{Car-M} can be easily adapted as below to reach the following conclusion. Actually, in the Gaussian space, a similar proof appears in \cite{C-K}. 
\begin{proposition}  
For all $A \subset M$, 
\begin{equation} \label{bord}
\mathrm{Per}_{\psi}(A) = \lim_{t \to 0} \int_{M} |\nabla P_t \mathbf{1}_A| d\mu.
\end{equation}
\end{proposition}

\begin{proof} [Proof]
First, as shown in \cite{Car-M}, one have 
\[
\mathrm{Per}_{\psi}(A) =  \sup \bigg \{ \int_{A}  (\mathrm{div} \varphi - \nabla \psi \cdot \varphi) d\mu, \, \varphi \in C^1(M, \R)\, \, \& \, \| \varphi \|_{\infty} \leq 1 \bigg \}.
\] 
Then one have, for all $\varphi \in C^1(M, \R)$ with $\| \varphi \|_{\infty} \leq 1$, by integration by parts, 
\begin{eqnarray*}
\int_A (\mathrm{div} \varphi - \nabla \psi \cdot \varphi ) d\mu & = & \lim_{t \to 0} \int_M P_t \mathbf{1}_A (\mathrm{div} \varphi - \nabla \psi \cdot \varphi ) d\mu \\
& = &  \lim_{t \to 0} \int_M \langle \nabla P_t \mathbf{1}_A,  \varphi \rangle d\mu  \leq  \lim_{t \to 0} \int_M |\nabla P_t \mathbf{1}_A | \ d\mu.
\end{eqnarray*}
Conversely, for all $\varphi \in C^1(M, \R)$ with $\| \varphi \|_{\infty} \leq 1$, by integrations by parts, and using the fact that $(P_t)_{t \geq 0}$ is self adjoint, 
\begin{eqnarray*}
\int_M \langle \nabla P_t \mathbf{1}_A,  \varphi \rangle d\mu  & = & \int_A P_t (\mathrm{div} \varphi - \nabla \psi \cdot \varphi ) d\mu  \\
& = & \int_A  (\mathrm{div} P_t (\varphi) - \nabla \psi \cdot P_t(\varphi) ) d\mu  \leq  \mathrm{Per}_{\psi}(A),
\end{eqnarray*}
the last inequality following from the fact that $\| P_t \varphi\|_{\infty} \leq 1$. Since 
\[
\int_M |\nabla P_t \mathbf{1}_A | \ d\mu = \mathrm{sup}_{\varphi, \, \| \varphi\|_{L^{\infty}} \leq 1} \int_M \langle \nabla P_t \mathbf{1}_A,  \varphi \rangle d\mu = \mathrm{sup}_{\varphi \in C^1(M, \R), \, \| \varphi\|_{L^{\infty}} \leq 1} \int_M \langle \nabla P_t \mathbf{1}_A,  \varphi \rangle d\mu, 
\]
it concludes the proof.

\end{proof}

As an immediate consequence, this definition of boundary is the limit in Bobkov semigroup functional, i.e.  
\[
\frac{1}{\sqrt{\kappa}} \mathrm{Per}_{\psi}(A) = \lim_{t \to 0} \int_{M} \sqrt{ \mathcal{I}_{\gamma}^2 (P_t \mathbf{1}_A) + \frac{1}{\kappa} |\nabla P_t \mathbf{1}_A |^2 } d\mu.
\]
Now define the isoperimetric problem with respect to this definition of boundary, i.e. define the following isoperimtric profile 
\[
\tilde{\mathcal{I}}_{(M,g, \mu)}(v) = \mathrm{inf}\{ \mathrm{Per}_{\psi}(A), \, A \subset M, \, \mu(A) = v \}.
\]
It is known that in general, for any Borel set $A \subset M$, $\mu^{+}(A) \geq \mathrm{Per}_{\psi}(A)$ so that this isoperimetric problem is \textit{a priori} weaker than the one we have formulated. However, with respect to this new notion of boundary, for any weighted manifold endowed with probability (finite) measure the infimum is a minimum. This fact is not hard to prove, see \cite{R-R}. Since for each $v$ there exists an extremal set of measure $v$, say $A$, by the above fact, $t \mapsto P_t \mathbf{1}_A$ has to be constant along Bobkov functional. This in turns implies rigidity by the preceding section.

It is well known that when the set $A$ is ``regular enough'' (say with $C^2$ smooth boundary $\partial A$), its relative perimeter agrees with its Minkowski content and both are equal to $ \int_{\partial A} \psi(x) d\mathcal{H}^{n-1} (x)$. In \cite{C-K}, the authors define indeed the Gaussian isoperimetric problem with the function $\tilde{\mathcal{I}}_{(\R^n, \gamma_n)}(v)$. In this space, the isoperimetric problems with respect to any kind of boundary agree. This property has been used by Carlen and Kerce, although somewhat hidden in \cite{C-K}. This enables them to establish full characterisation of isoperimetric sets, without any smoothness assumptions.  

Actually the conclusion remains true in the more general setting of weighted Riemannian manifold. That is, in weighted Riemannian manifold there is a minimizer such that $\mathrm{Per}_{\psi}(A)$ agrees with the Minkowski content $\mu^{+}(A)$ and moreover both are equal to $ \int_{\partial A} \psi(x) d\mathcal{H}^{n-1} (x)$. In fact, there exists a minimizer $A$ whose boundary $\partial A$ is the union of a regular part $\partial_r A$ and a closed set of singularity $\partial_s A$, empty is the dimension $n$ is less or equal than $7$, and of Hausdorff dimension at most $n-8$ if $n \geq 8$. This follows from the works on geometric measure theory initiated by Almgren, Morgan and others (see \cite{Mi1}, section $2.2$, for a comprehensive discussion on this topic). We refer to \cite{Mi1} for a complete bibliography. As a result, in this setting we have that $\tilde{\mathcal{I}}_{(M,g, \mu)} = \mathcal{I}_{(M,g, \mu)}$.

The above discussion implies the following fact. Let $(M, g, e^{-\psi} d\mathrm{vol})$ a weighted Riemmannian manifold of $CD(\kappa, \infty)$ class, that is $\mathrm{Ric}_g + \nabla^2 \psi \geq \kappa$. Assume that there is equality in one point $v \in (0,1)$ in Bakry--Ledoux comparison theorem. Then there is a set $A \subset M$ of measure $v$ achieving equality in Bobkov inequality. The previous sections imply therefore a splitting theorem. Moreover, in the Gaussian direction, $\Phi^{-1} \circ {Q_t \mathbf{1}_A}$ is a linear function. Following \cite{C-K}, this implies that $A$ must be an half space. In order to be self contained, let us recall their proof that also appears in \cite{M-N1}. 

\begin{lemma}
Let $A$ be a subset of $\R$ such that $\Phi^{-1} \circ {Q_t \mathbf{1}_A}$ is a linear function. Then $A$ is an half-space.
\end{lemma}

\begin{proof} [Proof]
Let $f_t$ to be $Q_t \mathbf{1}_A$. Then by hypothesis $
f_t = \Phi (\langle a_t, x + b_t \rangle).$ Now let $H =\{x \in \R^n, \langle a, x \rangle + b \geq 0 \}$ with $a \in \mathbb{S}^{n-1}, \, b \in \R$ such that $\gamma(H) = \gamma(A)$. Define $k_t = \frac{1}{\sqrt{1-e^{-2t}}}$. Then a fairly easy calculation shows that $Q_t \mathbf{1}_H = \Phi(\langle k_t a, x + b\rangle)$. If $k_t > |a_t|$, we can find a $s >0$ such that $k_{t+s} = |a_t|$ and then there is an half space $H_1$ such that $Q_{t+s} \mathbf{1}_{H_1} = Q_t \mathbf{1}_A$ so that $Q_s \mathbf{1}_{H_1} = \mathbf{1}_A$. This is impossible since $Q_t f$ is smooth for every $t >0$. Therefore $k_t \leq |a_t|$ and for the same smoothness assumptions necessary $k_t = |a_t|$. Then $A = H$ since the semigroup $(Q_t)_{t \geq 0}$ is one-to-one.

\end{proof}

Summarizing, we have the following theorem.

\begin{theorem}
Let $(M, g, \mu)$ a weighted Riemmannian manifold of $CD(\kappa, \infty)$ class. If there exists a $v \in (0,1)$ such that $\mathcal{I}_{(M, \mu)} (v) = \mathcal{I}_{(\R, \gamma_{\kappa})} (v)$ then $(M, g, \mu) = (\R, | \cdot |, \gamma_{\kappa}) \times (M', g',\mu')$. Moreover the isoperimetric minimizers are of the form $( - \infty, \Phi_{\kappa}^{-1}(v)) \times M'$.
\end{theorem}

This theorem is not new : indeed, it has been already proved by Franck Morgan \cite{Mo}) using more geometrical arguments. This proof only use semigroup arguments. 

Notice also that in the Gaussian space, the above lemma proves in the same manner that halfspaces are the only minimizers of the reverse Bobkov inequality \eqref{reverseBobkov}. However, by integration by parts, \eqref{reverseBobkov} expresses that halfspaces maximize the Euclidean norm of barycenters of sets of given measure. This fact is far more general and rather easy to prove (see \cite{B-C-F}). 

\section{Second order Poincar\'e type inequalities}

One can notice than in each of the proofs of Section~\ref{sec:geofun}, arises the common non negative factor $\Gamma_2 - \kappa \Gamma$. Thus, in view of studying robustness of geometric or functional inequalities, it would be worth to give a lower bound of the quantity $\int_{E} (\Gamma_2 - \kappa \Gamma)(f) d\mu$ in terms of a distance to extremal functions.

In this section, we derive a second order Poincar\'e type inequality for the spheres and a subclasses of log-concave measures. For sake of simplicity, we consider $CD(1, \infty)$ spaces although it works more generally for $CD(\kappa, \infty)$, $\kappa >0$. 

We have seen in Section~\ref{sec:stein} that the equality $\Gamma_2 (f) = \Gamma (f)$ imply $Lf + f =0$. With little additional effort, we can give a lower bound on $\int_E (\Gamma_2  - \Gamma) (f) d\mu$ in term of the $L^2$ distance of the operator $L + Id$. Indeed, let $(E, \mu, \Gamma)$ be a $CD(1, \infty)$ space. Fix a centered function $f$ on the Dirichlet domain. By the condition $\Gamma_2 \geq \Gamma$ and the spectral gap inequality, we get
\[
\int_E (Lf)^2 d\mu = \int_E \Gamma_2 (f) d\mu \geq \int_{E} \Gamma (f) d\mu \geq \int_{E} f^2 d\mu.
\] 
Thus
\begin{eqnarray*}
\int_E (\Gamma_2 - \Gamma) (f) d\mu & = & 
\int_E (Lf)^2 d\mu - \int_E f (-Lf) d\mu \\
& \geq &  \frac{1}{2} \bigg( \int_E  (Lf)^2 d\mu - 2 \int_E f (-Lf) + \int_E f^2 d\mu  \bigg)  \\
&= &  \frac{1}{2} \| f + Lf \|_{L^2(\mu)}^2. 
\end{eqnarray*}

As a standard example, let $(E, \mu, \Gamma)$ be the Gaussian space. Then, the eigenvectors of the operator $-L$ are given by the Hermite polynomials $h_{\alpha}$ associated to the eigenvalues $|\alpha | = \alpha_1 + \cdots + \alpha_n$. If $f_k = \sum_{\alpha, \, |\alpha| = k} \langle f, h_{\alpha} \rangle h_{\alpha}$ designs the chaos of order $k$, $k \in \N$, we have that  
\[
\| f + Lf \|_{L^2(\mu)}^2 = \sum_{k \geq 2} (k-1)^2  f_k ^2 \geq \sum_{k \geq 2}  f_k^2 = \| f - \Pi_1 f \|_2^2,
\] 
with $\Pi_1$ being the projection on linear functions. Since in the Gaussian space, $\Gamma_2 - \Gamma$ is simply the Hilbert-Schmidt norm of the Hessian, this implies in particular a second order Poincar\'e inequality. This inequality appears in a sharp form in \cite{M-N1} (i.e. without the factor $2$) and is used by the authors to prove robust dimension free Gaussian isoperimetry.

Recent results (\cite{CE}, \cite{Mi2}) allow us to extend this second order Poincar\'e inequality for more general log-concave measures with Hessian bounded from below and above i.e. $1 \leq \nabla^2 V \leq K$. Indeed following E. Milman's result, the spectrum $(\lambda_k)_{k \geq 0}$ of the operator $L = \Delta - \nabla V \cdot \nabla $ satisfies $\lambda_k \geq k$ for all $k \in \N$. As a result, $(\lambda_k-1)^2 \geq \frac{1}{2} \lambda_k$ for all $k \geq 2$ and so we have
\[
\| f + Lf \|_2^2  = \sum_{k \geq 2} (\lambda_k-1)^2 f_k^2 \geq \frac{1}{2} \sum_{k \geq 2} \lambda_k f_k^2  = \frac{1}{2} \left( \int_{\R^n} |\nabla f |^2 d\mu - \mathrm{Var}_{\mu} (f) \right).
\]
The right hand-side is simply the spectral gap deficit, say $\delta_{SG}$. As a consequence of a more general result concerning variance Brascamp--Lieb inequality of Cordero-Erausquin \cite{CE}, there exists $C = C(K) >0$ $v_0 \in \R^n$ and $t_0 \in \R$ such that $\delta_{SG} \geq C \|f - \nabla V \cdot v_0 + t_0 \|_2^2$ and so 
\[
\int_{\R^n} (\| \nabla^2 f \|_{HS}^2 + (\nabla^2 V - Id)(\nabla f, \nabla f) ) e^{- V(x)}dx = \int_{\R^n} (\Gamma_2 - \Gamma) (f) d\mu \geq C \|f - \nabla V \cdot v_0 + t_0 \|_2^2.
\] 
Here $t_0 = - \int_{\R^n} f d\mu = - \mathbb{E}_{\mu} (f)$ and $v_0 = \int_{\R^n} x (f(x) - \mathbb{E}_{\mu} (f) ) d\mu(x)$. This can be seen as a generalized second order Poincar\'e inequality.

Other interesting instances are given by the Euclidean spheres $(\mathbb{S}_1^n, g_1)$ satisfying the $CD(1,n)$ condition (that is of radius $\sqrt{n-1}$ with this normalization) with the associated Laplace--Beltrami operator $\Delta$. The eigenvalues of $- \Delta$ are given by $\lambda_k = k(n+k-1)/(n-1)$ and it is clear that whenever $k \geq 2$, $(\lambda_k-1)^2 \geq 1$. Thus
\[
\int_{\mathbb{S}_1^n} (\Gamma_2 - \Gamma) (f) d\mu  = \int_{\mathbb{S}_1^n} \| \nabla^2 f\|_{HS}^2 d\mu  \geq \frac{1}{2} \| f + \Delta f \|_2^2 \geq  \frac{1}{2}  \sum_{k \geq 2} f_k^2.
\]
Besides, it is well known that the corresponding eigenvectors $(Z_k)_{|k| \geq 1}$ are given by the spherical harmonics (see e.g. \cite{M-W}). The eigenvectors associated to $\lambda_1$ are simply given by the coordinates functions. As a consequence, similarly as in the Gaussian space, the right hand side corresponds to $\|f - \Pi_1 f\|_2^2$ where $\Pi_1 f$ is the projection of $f$ on linear functions. 

This extention of the Gaussian case to more general log-concave measures and the Euclidean spheres is important toward the generalization of the main result of Mossel and Neeman \cite{M-N1} about quantitative dimension free isoperimetry for the Gaussian measure. The next section is devoted to this problem.
 
\section{Quantitative estimates in the Bakry--Ledoux comparison theorem.} \label{sec:def}

Mossel and Neeman's scheme of proof is directly inspired by the heat-flow proof of Bobkov's inequality of Bakry and Ledoux. In Section~\ref{sec:geofun}, we have build the first step in expressing the derivative of the Bobkov functionnal for general spaces satisfying the $CD(1, \infty)$ condition in the same manner as in the Gaussian space. One can therefore ask if is it possible to push the arguments of \cite{M-N1} to more general $CD(1, \infty)$-spaces. On those spaces, cases of near equality have to imply that both near extremal sets are close to half spaces and that the subsequent space is close to the Gaussian space in an appropriate sense. There is some vagueness in this affirmation because it seems that there is no canonical way to define a good notion of closeness with respect to the Gaussian space.

In the case of log-concave probability measures, the isoperimetric deficit on the real line (where by the Bobkov's result extremals sets are half lines) appears in a sharp form in the work of de Castro \cite{DC} for more general measures than the ones satisfying the $CD(1, \infty)$ condition. On $\R^n$ ($n >1$) however, for other log-concave measures than the Gaussian, dimension free estimates are left open. As we saw in the preceding sections, when equality holds in the Bobkov inequality, we have that for every $t > 0$, $(\Gamma_2 - \Gamma) (\Phi^{-1} \circ P_t \mathbf{1}_A) = 0 $ so that it forces $A$ to be an half space. Besides, this implies that $\mu$ must have a one dimensional Gaussian marginal. It is a very natural guess that in the above setting, a subset $A \subset \R^n$ such that $\mu^{+}(A) = \mathcal{I}_{\gamma}(\mu(A)) + \delta$ must be ``close'' to an half space when $\delta$ is (very) small. 

An other interesting instance is given by the $n$ dimensional Euclidean spheres of radius $\sqrt{n-1}$ - denoted by $\mathbb{S}_1^n$ - equipped with the uniform measure for large values of $n$. Indeed, although there is no equality cases in the Bobkov inequality, the Poincar\'e lemma indicates that these spaces are approximately Gaussian.  Moreover, the isoperimetric sets are spherical caps and therefore of the same shape as the isoperimetric sets for the Gaussian space.

It is a consequence of \cite{Bar}, Proposition 11, that, point-wise, $\mathcal{I}_{\{\mathbb{S}_1^n, g_1, \mu\}}^{\frac{n}{n-1}} \leq \mathcal{I}_{\gamma}$. Therefore we get the estimates valid for all $v \in (0,1)$,
\[
\mathcal{I}_{\{\mathbb{S}_1^n, g_1, \mu\}}^{\frac{n}{n-1}}(v) \leq \mathcal{I}_{\gamma}(v) \leq \mathcal{I}_{\{\mathbb{S}_1^n, g_1, \mu\}} (v),
\]
and since for all $t \in (0, 1)$, $t -t^{\frac{n}{n-1}} = O(n^{-1})$, it implies that for all $v \in (0,1)$, 
\[
\mathcal{I}_{\{\mathbb{S}_1^n, g_1, \mu\}}(v) - \mathcal{I}_{\gamma}(v) = O(n^{-1}).
\] 
Toward robust isoperimetry, this is a good insight that one can consider it on the spheres $\mathbb{S}_1^n$ for some large $n$. 

It is well known that for small sets, the spherical isoperimetric profile is equivalent to the Euclidean one, i.e. to the map $x \mapsto x^{(n-1)/n}$ (up to some constant $a_n$), whereas the Gaussian isoperimetric profile behaves as the map $x \mapsto x \sqrt{-2 \log x}$. This indicates that the Bobkov inequality is a pretty bad approximation for the spherical isoperimetry for small (or large) sets. The optimal Bobkov inequality over $\mathbb{S}_1^n$, as recalled in Section~\ref{sec:infdim} (and after rescaling), is
\[
\mathcal{I}_{\gamma} \bigg(\int_{\mathbb{S}_1^n} f d\mu \bigg) \leq 
\int_{\mathbb{S}_1^n} \sqrt{ \mathcal{I}_{\gamma}^2 (f) + \frac{n-1}{c_n^2}|\nabla f |^2 } d\mu,
\]
where $c_n = \sqrt{2}\frac{\Gamma(\frac{n+1}{2})}{\Gamma (\frac{n}{2})}$. Using Stirling's formula, it is easily seen that $\frac{n-1}{c_n^2} = 1 - (2n)^{-1} + o (n^{-1})$, or equivalently $  \mathcal{I}_{\gamma}(1/2) = (1 - (2n)^{-1} + o (n^{-1})) \mathcal{I}_{\{\mathbb{S}_1^n, g_1, \mu\}} (1/2)$. Therefore, the preceding estimate $\mathcal{I}_{\{\mathbb{S}_1^n, g_1, \mu\}}(v) - \mathcal{I}_{\gamma}(v) = O(n^{-1})$ is tight.

In what follow, let $(E, \mu, \Gamma)$ be a $CD(1, \infty)$ space, and let $(P_t)_{t \geq 0}$ be its underlying semigroup.

\begin{hyp} \label{hypo}
Assume that there exists some positive constants $C, \eta, t_0$ and $\varepsilon_0 \in (0,1/2)$ such that, for all $t \in (0,t_0), \, \varepsilon \in (0, \varepsilon_0)$, the following upper-bound holds :  
\[
\int_{\{P_t f \leq \varepsilon \} } (\Gamma_2 - \Gamma) (\Phi^{-1} (P_t f)) d\mu \leq \frac{C}{t^\eta} \int_{\{P_t f \leq \varepsilon \}  }  (\Phi^{-1} (P_t f))^2 d\mu.
\]
\end{hyp}

This inequality may appear weird, but this is a kind of reverse isoperimetric inequality of second order. On the Gaussian space the point-wise estimate 
\[
(\Gamma_2 - \Gamma) (\Phi^{-1} (P_t f)) \leq  \frac{C(\Phi^{-1} (P_t f))^2}{t^\eta} 
\]
is satisfied with exponent $\eta = 2$ on the set $\{ P_t f \leq 1/2 \}$. Two different proofs has been given by Mossel and Neeman. Either directly by using the expression of the Ornstein-Uhlenbeck kernel (see \cite{M-N1}), either by pushing to the second order the arguments of the proof that amounts to the reverse Sobolev logarithmic inequality (see \cite{Nee}). However in the first case computations are specifics to the Ornstein-Uhlenbeck semigroup. In the second case, one need to use an exact commutation between $P_t$ and $\nabla$, which is no longer true for any other $CD(1, \infty)$ space than the Gaussian space. Concerning the log-concave setting $(\R^n, \mu)$, $d\mu(x) = e^{-V(x)}dx$, $1 \leq \nabla^2 V \leq K$, even under additional restrictions (such as $\Gamma_3  \geq 0$), it seems that we can reach a similar bound but crucially $C$ depends on the dimension $n$.

We will make use to the fact that $t$ goes to $0$, proving a similar point-wise upper bound - and therefore the validity of the hypothesis - for some $t_1$ and $\eta = 4$ in the case of the Euclidean sphere $\mathbb{S}_1^n$ endowed with uniform measure $\mu$. For that we will use short-times estimates on the spaces derivatives of the heat-kernel (see \cite{M-S}, \cite{Eng}). Such bounds are not yet proven in the non compact case. For sake of clarity we postpone the proof of this technical lemma to the appendix. 

\vspace{2mm}

We now state the main result of this section.

\begin{theorem} \label{thmrisosph}
Let $(\mathbb{S}_1^n,\mu)$ with $\mu$ the uniform probability measure. For each measurable subset $A$ of $\mathbb{S}_1^n$ such that
\[
\mu^+(A) = \mathcal{I}_{\gamma} (\mu(A)) + \delta,
\]
there exists a spherical cap $H$ and a positive constant $c \in (0.49, 1/2)$ such that  
\[
\mu(A \Delta H) \leq O( |\log \delta|^{-c}).
\]
\end{theorem}

Recall that $\mathcal{I}_{\gamma}$ and $\mathcal{I}_{\{ \mathbb{S}_1^n, g_1, \mu\}}$ differs only by $O(n^{-1})$. Therefore Theorem \ref{thmrisosph} implies to following following corollary for the deficit on the isoperimetric problem over the Euclidean spheres.

\begin{corollary} 
Let $(\mathbb{S}_1^n,\mu)$ with $\mu$ the uniform probability measure. Then for all measurable subset $A$ of $\mathbb{S}_1^n$ such that
\[
\mu^+(A) - \mathcal{I}_{(\mathbb{S}_1^n, g_1, \mu)} (\mu(A))  =   \delta,
\]
there exists a spherical cap $H$ and a positive constant $c \in (0.49, 1/2)$ such that  
\[
\mu(A \Delta H) \leq O( |\log \delta_n|^{-c}),
\]
where $\delta_n = \max(\delta, n^{-1})$.
\end{corollary}

The proof of this theorem is directly inspired from the work of \cite{M-N1}. We will use the general notation $(E, \mu, \Gamma)$, keeping in mind that most of the arguments are valid for general spaces. Specific bounds to the Euclidean spheres will be explicited in the process of the proof. Throughout the proof, $c, C$ will denote numerical constants (\textit{a priori} explicits) with $c \in (0,1)$ and $C \geq 1$ which may change from a line to another.

\begin{proof} [Proof of Theorem \ref{thmrisosph}]

The starting point is the semigroup proof of \cite{B-L1} of the functional version of the Bobkov inequality. We recall that if $f$ is the characteristic function of a measurable set $A$, then $\delta = \Psi(0) - \lim_{t \to \infty} \Psi(t) = \int_0^{\infty} - \Psi'(s) ds $, where 
$ \Psi \, :  t \in [0, \infty ) \mapsto  \int_{E}  \sqrt{ \mathcal{I}_{\gamma}^2 (P_t f) + \Gamma (P_t f)} d\mu.$
Moreover, by the remark of Section~\ref{sec:geofun}, for all $t >0$,
\[
- \Psi'(t) \geq  \int_{E} \frac{\mathcal{I}_{\gamma}(P_t f) (\Gamma_2 - \Gamma) (\Phi^{-1} \circ P_t f)}{(1 + \Gamma (\Phi^{-1} \circ P_t f))^{3/2}} d\mu,
\]
so that 
\begin{equation} \label{eq.defiso}
\delta \geq \mu^+(A) - \mathcal{I}_{\gamma} (\mu(A)) \geq  \int_0^{\infty} \int_{E} \frac{\mathcal{I}_{\gamma}(P_t \mathbf{1}_A) (\Gamma_2 - \Gamma) (\Phi^{-1} \circ P_t \mathbf{1}_A)}{(1 + \Gamma (\Phi^{-1} \circ P_t \mathbf{1}_A))^{3/2}} d\mu \, dt.
\end{equation}
Recall that, since $(P_t)_{t \geq 0}$ is mass preserving, and since the associated kernel $p_t$ is a positive function, for all $t > 0$, $P_t \mathbf{1}_A$ is a smooth function that takes value in $(0,1)$. Let now a smooth function $f \, : \, E \to (0,1)$ (that will be $P_s \mathbf{1}_A$ for some positive $s$) and $\delta$ to be the deficit in the Bobkov inequality associated to $f$.

In the preceding section, we gave a lower bound of $\int_{E} (\Gamma_2 - \Gamma) (f) d\mu$ in terms of projection over linear functions in the case of Euclidean spheres $\mathbb{S}_1^n$ or function of the type $x \mapsto \langle  a, \nabla V (x)  + b \rangle $ for log-concave measures on $\R^n$. In order to make use of this lower bound, the first task is to remove the quotient in the integrand of \eqref{eq.defiso}. For that, we make use of the reverse isoperimetric inequality given by Proposition~\ref{reverseiso}. Recall that for all function $f \, : \, E \to [0,1]$ and for all $t >0$, 
\[
C(t) \Gamma (\Phi^{-1} \circ P_t f)) \leq 1,
\] 
where  $C(t) = 2\int_0^t e^{2s} ds = e^{2t} -1$. 
Therefore, this implies 
\[
{(1 + \Gamma (\Phi^{-1} \circ P_t f))^{-3/2}} \geq  (1 + 1/C(t))^{-3/2}
\] and since $(1 + 1/C(t))^{-3/2} \sim (2t)^{-3/2} $ as $t$ goes to $0$,
for all $u >0$ and for some numerical constant $c$,
\[
\delta \geq \frac{c}{u^{-3/2}} \int_{u}^{2u} \int_{E} \mathcal{I}_{\gamma}(P_t f) (\Gamma_2 - \Gamma) (\Phi^{-1} \circ P_t f) d\mu \, dt.
\]
As a consequence, by the mean value theorem, there exists a $t \in (u,2u)$ such that
\begin{equation} \label{eqloca}
\delta \geq c u^{5/2} \int_{E} \mathcal{I}_{\gamma}(P_t f) (\Gamma_2 - \Gamma) (\Phi^{-1} \circ P_t f) d\mu \geq 2^{-5/2} c t^{5/2} \int_{E} \mathcal{I}_{\gamma}(P_t f) (\Gamma_2 - \Gamma) (\Phi^{-1} \circ P_t f) d\mu.
\end{equation}
All the remaining work is devoted to take account of the extra term $\mathcal{I}_{\gamma}(P_t f)$ (which is small whenever $P_t f$ is close to $0$ and $1$).

To get rid of this extra term, Mossel and Neeman use the reverse H\"older's inequality. After some work, this implies that $\Phi^{-1} (P_t f)$ is close to a linear function for some $t \in (c, c+1)$ when $c >0$ is universal. In order to conclude, the authors need a ``time reversal'' argument. For that they use spectral estimates that seem to be rather specific to the Gaussian space - even though it seems that it should work on the spheres. Our approach avoids this use by choosing an appropriate small $t$. Then the ``time reversal'' argument follows easily. Indeed, since we are close to achieve equality, $\| \sqrt{\Gamma (f)} \|_1$ is finite and so the reverse Poincar\'e inequality ensures that $f$ and $P_t f$ are close in $L^1$ distance. Besides, $\Phi$ is a Lipschitz map, so that $\Phi^{-1} (P_t f)$ and $\Phi^{-1} (f)$ are as least that close. More precisely, our proof continues as follows. 

Since  $x \mapsto \mathcal{I}_{\gamma}(x)$ is increasing on $(0,1/2)$ and symmetric with respect to $1/2$ on $(0,1)$, for $\varepsilon < 1/2$ we get that
\begin{eqnarray*}
   \int_{E} \mathcal{I}_{\gamma}(P_t f) (\Gamma_2 - \Gamma) (\Phi^{-1} \circ P_t f) d\mu 
& \geq &  \int_{\{\varepsilon \leq  P_t f \leq 1- \varepsilon\}} \mathcal{I}_{\gamma} (P_t f) (\Gamma_2 - \Gamma) (\Phi^{-1} \circ P_t f) d\mu  \\
& \geq & \mathcal{I}_{\gamma} (\varepsilon) \int_{\{ \varepsilon \leq  P_t f \leq 1- \varepsilon\}}  (\Gamma_2 - \Gamma) (\Phi^{-1} \circ P_t f) d\mu.
\end{eqnarray*}
Recalling \eqref{eqloca}, we deduce that
\begin{equation} \label{eqloceps}
\delta \geq c t^{5/2} \mathcal{I}_{\gamma} (\varepsilon) \int_{\{\varepsilon \leq  P_t f \leq 1- \varepsilon \}}  (\Gamma_2 - \Gamma) (\Phi^{-1} \circ P_t f) d\mu. 
\end{equation}
Moreover, denoting $h_t = \Phi^{-1} \circ P_t f$,
\begin{equation} \label{eqcut}
\int_{\{ \varepsilon \leq  P_t f \leq 1- \varepsilon \}}  (\Gamma_2 - \Gamma) (h_t) d\mu   
= \int_E  (\Gamma_2 - \Gamma) (h_t) d\mu   - \int_{\{ P_t f \leq \varepsilon \}}  (\Gamma_2 - \Gamma) (h_t) d\mu -  \int_{\{ P_t f \geq 1- \varepsilon \}}  (\Gamma_2 - \Gamma) (h_t) d\mu.
\end{equation}

The first term is precisely what we are looking for, as we proved a lower bound of it in previous section. We therefore need to show that, for an appropriate choice of $\varepsilon$, we do not lose too much while subtracting the integral on the sets $\{ P_t f \leq \varepsilon\}$ and $\{ P_t f \geq 1- \varepsilon \}$. That is why we made the hypothesis $(\mathcal{H})$. Recall that in the case of Euclidean spheres, $(\mathcal{H})$ is satisfied and it holds, for some $t_0 \in (0,1)$,
\[
\forall \varepsilon \leq 1/7 , \forall t \leq t_0, \,   \int_{\{ P_t f \leq \varepsilon\}}  (\Gamma_2 - \Gamma) (h_t) d\mu \leq \frac{C}{t^4} \int_{\{ P_t f \leq \varepsilon\}}   h_t^2 d\mu,
\]
so that by symmetry
\[
\forall \varepsilon \leq 1/7, \, \forall t \leq t_0, \,  \int_{\{ P_t f \geq 1 - \varepsilon\}}  (\Gamma_2 - \Gamma) (h_t) d\mu \leq \frac{C}{t^4} \int_{\{ P_t f \geq 1-\varepsilon\}}   (\Phi^{-1}(1- P_t f))^2 d\mu. 
\]
Both integrals are bounded in the same manner. We deal with the first one.  Recall that as a consequence of the reverse isoperimetric inequality $h_t$ is $(2t)^{-1/2}$-Lipschitz and so by classical concentration for Lipschitz maps, $\mu ( \{ -h_t \geq u \}) \leq e^{-tu^2}$. Since 
\[
\int_{\{ P_t f \leq \varepsilon\}}   h_t^2 d\mu  = \int_{\{-h_t \geq -\Phi^{-1}(\varepsilon)\}} h_t^2 d\mu,
\]
by the layer-cake representation,
\[
\int_{\{-h_t \geq -\Phi^{-1}(\varepsilon)\}}  h_t^2 d\mu  = 2  \int_{-\Phi^{-1}(\varepsilon)}^{\infty} u \mu ( \{ -h_t \geq u \}) du \leq  2 \int_{-\Phi^{-1}(\varepsilon)}^{\infty} u e^{-tu^2} du = \frac{ e^{-  t (\Phi^{-1}(\varepsilon))^2 }}{t}.
\]
Thus 
\begin{equation} \label{eqconc}
\int_{\{ P_t f \leq \varepsilon\}}  (\Gamma_2 - \Gamma) (h_t) d\mu + \int_{\{ P_t f \geq 1 - \varepsilon\}}  (\Gamma_2 - \Gamma) (h_t) d\mu \leq \frac{C}{t^5} e^{- t (\Phi^{-1}(\varepsilon))^2 }.
\end{equation}
Therefore, using \eqref{eqloceps}, \eqref{eqcut}, and \eqref{eqconc}, we get 
\[
\delta \geq c t^{5/2} \mathcal{I}_{\gamma}(\varepsilon) \bigg( \int_{E}  (\Gamma_2 - \Gamma) (h_t) d\mu  - \frac{1}{t^5} e^{-  t(\Phi^{-1}(\varepsilon))^2 } \bigg),
\]
so that
\[
\int_{E}  (\Gamma_2 - \Gamma) (h_t) d\mu \leq C \bigg( \frac{\delta}{ t^{5/2} \mathcal{I}_{\gamma}(\varepsilon)} + \frac{1}{t^5} e^{-  t (\Phi^{-1}(\varepsilon))^2 } \bigg).
\]

\medskip

Besides, recall that the Poincar\'e inequality of second order for $(E, \mu, \Gamma) = (\mathbb{S}_1^n, \mu)$ implies that
\[
\|h_t - \Pi_1 h_t\|_2^2 \leq  \int_{E}  (\Gamma_2 - \Gamma) (h_t) d\mu,
\]
where $\Pi_1$ is the projection on linear functions. In this case, the two preceding bounds imply then 
\begin{equation} \label{eqPhi}
\|h_t - \Pi_1 h_t\|_2^2  \leq C \bigg( \frac{\delta}{ \mathcal{I}_{\gamma}(\varepsilon) t^{5/2} } + \frac{1}{t^5} e^{-  t (\Phi^{-1}(\varepsilon))^2 } \bigg).
\end{equation}

\vspace{2mm}

Now, we recall that as a consequence of the reverse Poincar\'e inequality, \eqref{reversesg} holds, i.e. 
\[
\| f - P_t f \|_1 \leq \sqrt{2t} \| \sqrt{\Gamma (f)} \|_1 .
\] 
Since we are close to a case of equality, we can make the assumption that 
\[
\| \sqrt{\Gamma(f)} \|_1 \leq \mathcal{I}_{\gamma} (\mathbb{E} f) + \delta \leq \frac{1}{2}
\]
otherwise $\delta$ is larger than $1/2 - (2\pi)^{-1/2}$ and the announced claim is still true. This implies that $\|f - P_t f\|_1 \leq \sqrt{2t}/2$. Using the triangular inequality, we thus get 
\begin{equation} \label{eqtri}
\|f - \Phi(\Pi_1 h_t) \|_{L^1(\mu)} \leq   \frac{\sqrt{2}}{2} \sqrt{t} + \|P_t f - \Phi(\Pi_1 h_t)\|_1.
\end{equation}
Besides, since $\Phi$ is $1$-Lipschitz, for every norm $\| \, \cdot \, \|$, it holds 
\[
\|P_t f - \Phi(\Pi_1 h_t)\| = \|\Phi (h_t)  - \Phi(\Pi_1 h_t)\| \leq  \|h_t - \Pi_1 h_t\|,
\] 
and therefore
\begin{equation} \label{eqcontr}
\|P_t f - \Phi(\Pi_1 h_t)\|_1 \leq  \|P_t f - \Phi(\Pi_1 h_t)\|_2 \leq \|h_t - \Pi_1 h_t \|_2.
\end{equation}
Using \eqref{eqPhi}, \eqref{eqtri} and \eqref{eqcontr} together, we obtain 
\[
\|f - \Phi(\Pi_1 h_t)\|_1 
\leq \frac{\sqrt{2}}{2}\sqrt{t} + \ C \bigg( \frac{\delta}{ t^{5/2}  \mathcal{I}_{\gamma}(\varepsilon)} + \frac{1}{t^5} e^{-  t  (\Phi^{-1}(\varepsilon))^2} \bigg)^{1/2}.
\]

Recall that $\varepsilon \in (0, 1/7)$ and $t \in (0,t_0)$ are arbitrary. It remains to optimize over $t$ and $\varepsilon$. We use the well known asymptotic estimates $\mathcal{I}_{\gamma} (x) \sim x \sqrt{-2 \log x}$ and $\mathcal{I}'_{\gamma} (x) = \Phi^{-1}(x) \sim -\sqrt{-2 \log x}$, when $x$ goes to $0$.

Therefore, we can take $t = |\log (\delta)|^{-2c},$ for some $c \in (0.49,1/2)$, and $\varepsilon = \delta^{1/2}$. The second term between the brackets decays at faster rate in $\delta$ than the first one so that it yields the existence of a positive constant $\delta_0$ such that
\begin{equation} \label{nearsets}
\forall \delta \leq \delta_0, \, \|f - \Phi(\Pi_1 h_t)\|_{L^1(\mu)} \leq  |\log \delta|^{-c}.
\end{equation}

%Notice that $t$ depends on $\delta$. Moreover, since  
%\[ \nabla \Phi(\nabla V \cdot v_0 + a_0) = \varphi (\nabla V \cdot v_0 + a_0) \nabla^2 V (a_0), \]
%by the assumption $\nabla^2 V \leq K$, we get still using \eqref{reversesg},
%\[\|\Phi(\nabla V \cdot v_0 + a_0) - \Phi(\nabla V \cdot v_t + a_t)\|_{L^1(\mu)} \leq K |a_0|^2 \sqrt{t} = K |a_0|^2 |\log(\delta)|^{-c} = O_{K, \mu(A)} (|\log(\delta)|^{-c}),\]
%so that 
%\[\forall \delta \leq \delta_0, \, \|f -  \Phi(\nabla V \cdot v_0 + a_0) \|_1 \leq O (|\log \delta|^{-c}).\]
To recover the statement of Theorem \ref{thmrisolg}, we now need to go back to sets. The following arguments are directly taken from \cite{M-N1}.

First, we apply this result for some smooth approximation $f$ of $\mathbf{1}_A$. Using strong continuity of $(P_t)_{t \geq 0}$, we can find a small $s >0$ such that for every $g \in L^1(E),$
\[
\|\mathbf{1}_A - g \|_1 \leq \| P_s \mathbf{1}_A -  g\|_1 + \delta.
\]
Since $P_s \mathbf{1}_A$ is smooth for every $s \geq 0$, we can apply \eqref{nearsets} for $f = P_s \mathbf{1}_A$ so that 
\[
\forall \delta \leq \delta_0, \, \|\mathbf{1}_A -  \Phi(\Pi_1 h_t) \|_1 \leq \|f -  \Phi(\Pi_1 h_t) \|_1 + \delta \leq 2 |\log \delta|^{-c}.
\] 
In order to conclude, we use the following lemma which is a straightforward generalization of Lemma 5.2 from \cite{M-N1} - it consists of rounding $\Phi(\Pi_1 h_t)$ to $\{0,1\}$.

\begin{lemma}
Let $A$ be measurable set and $g = \Phi(\Pi_1 h_t)$. Let $H$ to be the spherical cap $\{\Pi_1 h_t \geq 0 \}$. Then 
\[
\mu(A \Delta H) =  \int_{E} |\mathbf{1}_A - \mathbf{1}_H| d\mu \leq \int_{E} |\mathbf{1}_A - g| d\mu.
\]
\end{lemma}

\noindent Obviously, if $\delta \geq \delta_0$, $\mu(A \Delta H) \leq |\log \delta_0|^c |\log \delta|^{-c}$ so that for every $\delta \geq 0$, 
\[
\mu(A \Delta H) \leq O(|\log \delta|^{-c}).
\] 
Theorem \ref{thmrisosph} is thus established.  
\end{proof}

\vspace{2mm}

\noindent \textit{Remark} : Concerning the log-concave case, we have seen that the second order Poincar\'e's inequality reads as 
\[
 \int_{\R^n} (\Gamma_2 - \Gamma) (f) d\mu \geq C \|f - \nabla V \cdot v_0 + t_0 \|_2^2,
\] 
with  $t_0 = - \int_{\R^n} f d\mu = - \mathbb{E} f$ and $v_0 = \int_{\R^n} x (f(x) - \mathbb{E} f ) d\mu(x)$. 
If Hypothesis \eqref{hypo} holds (with $C$ independent on $n$), it would imply by a completely similar proof the following Theorem.   
\begin{theorem} \label{thmrisolg}
Let $(\R^n,d\mu(x) = e^{-V(x)}dx)$ where $\mu$ is such that $1 \leq \nabla^2 V \leq K$. Let $A$ be a set such that 
\[
\mu^+(A) \leq \mathcal{I}_{\gamma} (\mu(A)) + \delta.
\]
Then there exists a set $B$ with of the form $B = \{x \in \R^n,  \nabla V(x) \cdot v_{0} + x_{0} \geq 0 \},$  with some $v_0 \in \R^n$ and $x_0 \in \R$, and a positive constant $c$ such that  
\[
\mu(A \Delta B) \leq O( |\log \delta|^{-c}).
\]
\end{theorem}

\section{Appendix : proof of Hypothesis \ref{hypo} in the case of the Euclidean spheres.}

In this appendix, we prove the following technical lemma, in the case of Euclidean spheres $\mathbb{S}_1^n$. 
\begin{lemma}  \label{lemmatough}
There exists some positive constants $t_0 \in (0,1)$, $C$ such that, for all $0 < t \leq t_0$ and all $\varepsilon \in (0,1/7)$, 
\[
\int_{\{ P_t f \leq \varepsilon \} } (\Gamma_2 - \Gamma) (\Phi^{-1} (P_t f)) \, d\mu \leq  \frac{C}{t^4}  \int_{\{ P_t f \leq \varepsilon\}}   (\Phi^{-1} (P_t f))^2 d\mu.
\]
\end{lemma}
We shall show that the upper bound holds point-wise on the set $\{ P_t f \leq \varepsilon \}$ for sufficiently small $t$ and $\varepsilon$. We will write as a short-hand $f_t$ for $P_t f$ and $h_t$ for $\Phi^{-1} (P_t f)$ and stick to abstract notations $(E, \mu, \Gamma)$ since most of the arguments are valids in general.

\begin{proof} [Proof of Lemma \ref{lemmatough}]

Recall that for Euclidean Spheres, 
\begin{equation} \label{eq.gamma2}
(\Gamma_2 - \Gamma)(h_t) = \| \nabla^2 h_t\|_{HS}^2.
\end{equation}
First, we expand the hessian thanks to the chain rule formula.
\begin{eqnarray*}
\| \nabla^2 h_t \|_{HS}^2 & = & \bigg\| \frac{\nabla^2 f_t}{\mathcal{I}_{\gamma}(f_t)}  +  \frac{\mathcal{I}_{\gamma}'(f_t) \nabla f_t \, ^T \nabla f_t }{\mathcal{I}_{\gamma}(f_t)^2} \bigg\|_{HS}^2 \\
& = & \bigg\| \frac{\nabla^2 f_t}{\mathcal{I}_{\gamma}(f_t)}  +  \mathcal{I}_{\gamma}'(f_t) \nabla h_t \, ^T \nabla h_t  \bigg\|_{HS}^2.
\end{eqnarray*} 
Using the inequality $(a+b)^2 \leq 2(a^2+b^2)$, we get 
\begin{equation} \label{eq.hess}
\| \nabla^2 h_t \|_{HS}^2 \leq 2 \frac{\|\nabla^2 f_t\|_{HS}^2} {\mathcal{I}_{\gamma}^2(f_t)} + 2 (\mathcal{I}_{\gamma}'(f_t))^2  \| \nabla h_t \, ^T \nabla h_t \|_{HS}^2.
 \end{equation}
We will bound the two terms separately.

\vspace{2mm}

\noindent Recall that the reverse isoperimetric inequality reads as $C(t) \Gamma(P_t f)   \leq  \mathcal{I}_{\gamma}^2 (P_t f),$ or equivalently by the chain rule formula,
\begin{equation} \label{eq.reverseiso}
C(t) \Gamma(h_t) \leq 1.
\end{equation}
Since $t$ is small, we can simply take $C(t) = 2t$ as $C(t) \sim 2t$ when $t$ goes to 0. Thus, \eqref{eq.reverseiso} implies that
\begin{equation} \label{order1}
(\mathcal{I}_{\gamma}'(f_t))^2  \| \nabla h_t \, ^T \nabla h_t \|_{HS}^2 = (\mathcal{I}_{\gamma}'(f_t))^2 (\Gamma (h_t))^2 \leq \frac{(\mathcal{I}_{\gamma}'(f_t))^2}{4t^2} =  \frac{h_t^2}{4t^2}.
\end{equation}

Recalling \eqref{eq.gamma2} and \eqref{eq.hess}, by \eqref{order1}, in view of Lemma \ref{lemmatough}, we will prove 
\[
\exists t_0 \in (0,1), \, \, \forall t \in (0, t_0), \, \frac{\|\nabla^2 f_t\|_{HS}^2} {\mathcal{I}_{\gamma}^2(f_t)} \leq \frac{C h_t^2}{t^4}
\]
on the set $\{f_t \leq 1/7 \}$.
\medskip

We therefore proceed to show that actually a stronger estimates holds. It is known indeed that the following estimates on the derivatives of the heat kernel $p_t(x,y)$ associated to the heat semi-group hold - these estimates are valid for compact Riemannian manifolds with Ricci curvature bounded from below (see e.g \cite{M-S}, \cite{Eng}). 
\begin{lemma} \label{lemhkesti}
There exists a positive constant $C$ such that 
\begin{equation} 
\forall t \leq 1, \, | \nabla \log p_t (x,y) |^2 \leq \frac{C (1 + d^2(x,y))}{t^2}, \, \| \nabla^2 \log p_t (x,y) \|^2_{HS} \leq \frac{C(1 + d^2(x,y))}{t^4}.
\end{equation}
\end{lemma}
Such bounds in $(\R^n, \mu)$ satisfying $CD(\kappa, \infty)$ are not proven yet in general (see \cite{Li} for related work). Actually, as stated in \cite{M-S}, \cite{Eng}, $C$ depends on $n$. However, looking closely at the proofs, the dependance on $n$ for general compact Riemannian manifolds, comes from the lower and upper bounds on the heat kernel of the form 
\begin{equation} \label{hklup}
 \frac{c_1}{\sqrt{\mu(B(x, \sqrt{t})) \mu (B(y, \sqrt{t}))}} e^{- C_1 \frac{d(x,y)^2}{t}} \leq p_t(x,y) \leq \frac{C_2}{\sqrt{\mu(B(x, \sqrt{t})) \mu (B(y, \sqrt{t}))}} e^{- c_2 \frac{d(x,y)^2}{t}},
\end{equation}
where $B( \cdot , r)$ designs the geodesic ball centered in $\cdot$ of radius $r$. In general, if the curvature is simply bounded from below, $\mu(B( \cdot , \sqrt{t}))$ is estimated as approximately $t^{n/2}$ as in the Euclidean space. Yet, in the cases of $\mathbb{S}_1^n$, the uniform distribution $\mu$ is approximately the Gaussian measure $\gamma_1$ (by the Poincar\'e lemma). More precisely, by rotational invariance of the measure $\mu$ and the Pythagorean theorem,
\begin{eqnarray*}
\forall \theta \in \mathbb{S}_1^n, \, \mu(B( \theta , \sqrt{t})) &= & \frac{\mathrm{vol}_{n-1} \mathbb{S}^{n-1}}{\mathrm{vol}_n \mathbb{S}^n (n-1)^{n/2}}\int_{0}^{\sqrt{t}} (n-1 - u^2)^{\frac{n-1}{2}} \frac{du}{\sqrt{1-\frac{u^2}{n-1}}} \\
& = & \frac{1}{\sqrt{2\pi}} \frac{c_n}{\sqrt{n-1}} \int_0^{\sqrt{t}} \bigg ( 1 - \frac{u^2}{n-1} \bigg)^{\frac{n-2}{2}} \, du.
\end{eqnarray*}
It is not hard to check that, uniformly on $u >0$, 
\[
\frac{1}{\sqrt{2\pi}} \bigg| \frac{c_n}{\sqrt{n-1}}  \bigg( 1 - \frac{u^2}{n-1} \bigg)^{\frac{n-2}{2}}  - e^{-u^2/2} \bigg| \leq O (n^{-1})
\]
so that by integration $|\mu(B( \cdot , \sqrt{t})) - \gamma_1 ([0, \sqrt{t}])| = O (\sqrt{t} n^{-1})$. Thus the fractions in the estimates of \eqref{hklup} can be lower or upper bounded by dimension free ones. As a result (see \cite{Eng} for the details of how to reach from \eqref{hklup} the inequalities of Lemma \ref{lemhkesti}), the dependance on the dimension in $C$ can be remove. \footnote{Notice that in the log-concave case, even if an analogous statement of Lemma  \ref{lemhkesti} would be established, it would not be possible to remove the dependance on the dimension. Indeed, in the simplest case of the Ornstein--Uhlenbeck kernel, it is immediately checked by the Mehler's representation formula that these constants are of the form $C n$, where $n$ is the dimension.}

Expanding the derivatives, Lemma \ref{lemhkesti} can be restated as  
\[
\sum_{1 \leq i,j \leq n} \bigg( \frac{\partial_{x_i, x_j} p_t(x,y)}{p_t(x,y)} - \frac{\partial_{x_i} p_t(x,y) \partial_{x_j} p_t(x,y)}{p_t(x,y)^2} \bigg)^2 \leq \frac{C(1+d^2(x,y))}{t^4}
\]
and 
\[
 \sum_{1 \leq i,j \leq n} \bigg(\frac{\partial_{x_i} p_t(x,y) \partial_{x_j} p_t(x,y)}{p_t(x,y)^2} \bigg)^2 = 
\bigg( \sum_{i=1}^n \frac{(\partial_{x_i} p_t(x,y))^2}{p_t(x,y)^2} \bigg)^2 \leq \bigg( \frac{C(1+d^2(x,y))}{t^2} \bigg)^2.
\]
These two bounds imply, using the inequality $A^2 \leq 2[(A-B)^2 + B^2]$, 
\begin{equation} \label{hkest}
\sum_{1 \leq i,j \leq n}  \bigg( \frac{\partial_{x_i, x_j} p_t(x,y)}{p_t (x,y)} \bigg)^2  \leq \frac{C(1+d^2(x,y) + d^4(x,y))}{t^4}.
\end{equation}

Denoting $ d\tilde{\mu} (y) = \frac{f(y) p_t(x,y) d\mu(y)}{\int_E f(y) p_t(x,y) d\mu} = \frac{f(y) p_t(x,y) d\mu(y)}{f_t (x)},$ by Jensen's inequality, we have that 
\begin{eqnarray*}
\frac{1}{f_t^2 (x)} \sum_{1 \leq i,j \leq n} \bigg( \int_E f(y) \partial_{x_i, x_j} p_t (x,y) d\mu \bigg)^2 &= & \sum_{1 \leq i,j \leq n}  \bigg( \int_{E} \frac{\partial_{x_i, x_j} p_t (x,y)}{p_t(x,y)} d\tilde{\mu} \bigg)^2 \\
& \leq & \sum_{1 \leq i,j \leq n} \int_{E} \bigg( \frac{\partial_{x_i, x_j} p_t (x,y)}{p_t(x,y)}  \bigg) ^2 d\tilde{\mu}.
\end{eqnarray*}

\noindent Besides, \eqref{hkest} implies that
\[
\sum_{1 \leq i,j \leq n} \int_{E} \bigg( \frac{\partial_{x_i, x_j} p_t (x,y)}{p_t(x,y)}  \bigg) ^2 d\tilde{\mu} \leq \frac{C}{t^4} \bigg( 1  + \int_E  d(x,y)^2 (1 + d(x,y)^2) d\tilde{\mu} \bigg).
\]
Moreover,
\[
\int_E  d(x,y)^2 (1+d(x,y)^2) d\tilde{\mu} = \frac{1}{f_t} \int_{E} f(y) d(x,y)^2 (1+d(x,y)^2) p_t(x,y) d\mu(y), 
\]
and since when $t$ goes to $0$, $p_t(x,y)$ is approaching the Dirac mass $\delta_x (y)$, this integral goes to $0$ as $t$ goes to $0$. Therefore, there exists $T_1 >0$ such that for all $t \leq T_1$, $ \int_E  d(x,y)^2 (1+d(x,y)^2) d\tilde{\mu} \leq 1$,
and therefore for all $t \leq t_1 = \mathrm{min}(1,T_1)$ and all $x \in \mathbb{S}_1^n$,
\[
\frac{1}{f_t^2 (x)} \sum_{1 \leq i,j \leq n} \bigg( \int_E f(y) \partial_{x_i, x_j} p_t (x,y) d\mu \bigg)^2 \leq \frac{C}{t^4},
\]
which can be rewritten as $ \|\nabla^2 f_t \|_{HS}^2 \leq \frac{Cf_t^2}{t^4}.$ Then the point-wise estimates holds 
\[
\forall t \in (0,t_1), \, \frac{\|\nabla^2 f_t\|_{HS}^2} {\mathcal{I}^2_{\gamma}(f_t)} \leq \frac{C f_t^2}{t^4 \mathcal{I}_{\gamma}^2 (f_t)}.
\]
Notice that, on the set $\{ f_t \leq \varepsilon \}$, $\frac{f_t^2}{\mathcal{I}_{\gamma}^2 (f_t)}$ goes to $0$ as $\varepsilon$ goes to $0$ so that it actually holds a (much) better estimate than the one needed for Lemma \ref{lemmatough}. Anyway, since $h_t \geq 1$ as soon as $P_t f \leq \Phi(-1)$ and since $\Phi(-1) \geq 1/7$, it proves that for all $\varepsilon \leq 1/7$,
\[
\bigg\| \frac{\nabla^2 f_t}{\mathcal{I}_{\gamma}(f_t)} \bigg  \|_{HS}^2 \leq \frac{C h_t^2}{t^4}.   
\]
Thus, recalling \eqref{eq.hess} and \eqref{order1}, the following point-wise upper bound
\[
\forall t \in (0,t_0), \, \forall \varepsilon \in (0, 1/7), \quad (\Gamma_2 - \Gamma) (h_t) \leq \frac{C h_t^2}{t^4} 
\]
holds and after integration it yields Lemma \ref{lemmatough} for $(\mathbb{S}_1^n, \mu)$.

\end{proof}

\vskip 5 mm 

\textit{Acknowledgment. The main part of this work has been completed when I made my Ph.D at the University of Toulouse. I thank my Ph.D advisor Michel Ledoux for many valuable exchanges.}

\vskip 10 mm

\noindent
\textsc{Rapha{\"e}l Bouyrie,} \\
\textsc{\small{Laboratoire d'Analyse de Math\'ematiques Appliqu\'es, UMR 8050 du CNRS, Universit\'e Paris-Est Marne-la-Vall\'ee, 5 Bd Descartes, Champs-sur-Marne, 77454 Marne-la-Vall\'ee Cedex, France}} \\
\textit{E-mail address:} \texttt{raphael.bouyrie@gmail.com}


\begin{thebibliography} {cc}


\bibitem [Bak] {Ba} {\sc  D. Bakry} \emph{On Sobolev and Sobolev logarithmic inequalities for Markov semigroups}, \rm{New trends in stochastic analysis, 43–75, River Edge, NJ} (1997)


\bibitem [B-E] {B-E} {\sc  D. Bakry, M. Emery} \emph{Diffusions hypercontractives}, \rm{Séminaire de Probabilités XIX, Lecture Notes in Math. Springer-Verlag, New York, 1123, 177-206} (1985)

\bibitem [B-G-L] {B-G-L}
{\sc D. Bakry, I. Gentil, M. Ledoux.} \emph {Analysis and geometry of Markov diffusion operators} \rm{Grundlehren der mathematischen Wissenschaften~348. Springer} (2014)

\bibitem [B-L1] {B-L1}
{\sc D. Bakry,  M. Ledoux.} \emph {Levy--Gromov comparison Theorem for an infinite dimensional generators.} \rm{Invent. Math., 123, 259-281} (1996)

\bibitem [B-L2] {B-L2}
{\sc D. Bakry,  M. Ledoux.} \emph {Sobolev inequalities and Myers's diameter theorem for an abstract Markov generator.} \rm{Duke Math. J., l85, 1, 253-270} (1996)


\bibitem [B-B-J] {B-B-J} {\sc M. Barchiesi, A. Brancolini, V. Julin} \emph{Sharp dimension free quantitative estimates for the Gaussian isoperimetric inequality} \rm{Ann. of Probab,      45, 2, 668-697} (2017)

\bibitem [Bar] {Bar} {\sc  F. Barthe} \emph{Extremal Properties of Central Half-Spaces for Product Measures}, \rm{J. of Functional Analysis, 182,  81-107} (2001)


\bibitem [B-C-F] {B-C-F} {\sc F. Barthe, D. Cordero-Erausquin, M. Fradelizi.} \emph{Shift inequalities of Gaussian type and norms of barycenters} \rm{Studia Math. 146, 245–259} (2001) 

\bibitem [B-M] {B-M}{\sc F. Barthe,  B. Maurey.} \emph {Some remarks on isoperimetry of Gaussian type.} \rm{Ann. Inst. Henri Poincar\'e. Probab. Stat.  36, 419-434} (2000)


\bibitem [Bay] {Bay} {\sc  V. Bayle} \emph{Propriétés de concavité du profil isopérimétrique et applications}, \rm{Ph. Dissertation, Institut Joseph Fourier, Université de Grenoble} (2004)

\bibitem [Ben] {Ben}
{\sc A. Bentaleb.} \emph {Sur les fonctions extrémales des inégalités de Sobolev des opérateurs de diffusion.} \rm{Séminaire de probabilités de Strasbourg 36, 230-250} (2002)

\bibitem [Bo1] {Bo1}
{\sc S. Bobkov.} \emph {Extremal properties of half-spaces for log-concave distributions.} \rm{Ann. of Probab. 24, 35-48} (1996)

\bibitem [Bo2] {Bo2}
{\sc S. Bobkov.} \emph {An isoperimetric inequality on the discrete cube, and an elementary proof of the isoperimetric inequality in Gauss space.} \rm{Ann. of Probab. 25, 206-214} (1997)

\bibitem [C-M] {Car-M}  A. \sc Carbonaro, G. Mauceri  \emph{A note on bounded variation and heat semigroup on Riemannian manifolds}, \rm{Bull. Austral. Math. Soc. Vol. 76   155-160} (2007)

\bibitem [Car] {C}  E. \sc Carlen, \emph{Some integral identities and inequalities for entire functions and their application to the coherent state transform}, \rm{J. Funct. Anal. 97,  231–249} (1991)

\bibitem [C-K] {C-K}  E. \sc Carlen, C. \sc Kerce, \emph{On	the	Cases of Equality in Bobkov's Inequality and	Gaussian Rearrangement}, \rm{Calculus of Variations and Partial Differential Equations, 13, 1–18} (2001)

\bibitem [C-M1] {C-M1}  F. \sc Cavalletti,  A. \sc Mondino \emph{Sharp geometric and functional inequalities in metric-mesaure spaces with lower Ricci curvatiure bounds}, \rm{to appear in Invent. Math.} 

\bibitem [C-M2] {C-M2}  F. \sc Cavalletti,  A. \sc Mondino \emph{Sharp and rigid isoperimetric inequalities in metric-mesaure spaces with lower Ricci curvatiure bounds}, \rm{to appear in Geom. Topol.} 

\bibitem [CE] {CE}  D. \sc Cordero-Erausquin,  \emph{Transport inequalities for log-concave measures, quantitative forms and applications}, \rm{Can. J. of Math.} (2015) 

\bibitem [DC] {DC}  Y. \sc De Castro,  \emph{Quantitative Isoperimetric Inequalities on the Real Line}, \rm{Ann.
Math. Blaise Pascal 18 251–271} (2011) 

\bibitem [DG] {DG} E. \sc De Giorgi, \emph{Frontiere orientate di misura minima,} \rm{Seminario di Matematica della Scuola
Normale Superiore di Pisa, 1960-61} (1961)

\bibitem [Eld] {Eld}  R. \sc Eldan,  \emph{A two-sided estimate for the Gaussian noise stability deficit}, \rm{Invent. Math, 201, 561-624} (2015) 

\bibitem [Eng] {Eng}  A. \sc Engoulatov,  \emph{A universal bound on the gradient of logarithm of the heat kernel for manifolds with bounded Ricci curvature.} \rm{J. Funct. Anal. 238  no. 2, 518–529} (2006)

\bibitem [Kla] {Kl} B. \sc Klartag, \emph{Needle decomposition in Riemmanian geometry} \rm{Amer. Math. Soc., 249, 1180} (2017)

\bibitem [Gro] {G} \sc  L. Gross \emph{Logarithmic Sobolev inequalities}, \rm{Amer. J. Math., 97: 1061 - 1083} (1975) 

\bibitem [Le1] {Le1} \sc  M. Ledoux \emph{The geometry of Markov diffusion generators.}, \rm{Ann. Fac. Sci., Toulouse IX, 305-366}, (2000)

\bibitem [Le2] {Le2} \sc  M. Ledoux \emph{Chaos of a Markov operator and the fourth moment condition.}, \rm{Ann. of Probab. 40, 2439-2459} (2012)

\bibitem [Li] {Li} \sc  X.D. Li \emph{Hamilton's Harnack inequality and the $W$-entropy formula on complete Riemannian manifold.}, \rm{preprint} (2014)

\bibitem [M-S] {M-S} \sc   P. Mallavin, D.W. Stroock \emph{Short time behavior of the heat kernel and its logarithmic derivatives} \rm{Journ. of Diff. Geometry Vol 44, 550-570} (1996)

\bibitem [Mi1] {Mi1} \sc   E. Milman, \emph{Sharp Isoperimetric inequalities and model spaces for Curvature-Dimension-Diameter condition} \rm{J. Eur. Math. Soc. 17, 1041–1078} (2015)

\bibitem [Mi2] {Mi2} \sc   E. Milman, \emph{Spectral estimates, contractions and hypercontractivity} \rm{to appear in J. Spectr. Theory} (2015)

\bibitem [M-N] {M-N1} \sc   E. Mossel, J. Neeman \emph{Robust dimension free Gaussian isoperimetry} \rm{Ann. of Probab. 43, p.971-991} (2015)

\bibitem [Mor] {Mo} \sc   F. Morgan, \emph{Manifold with densities} \rm{ Notices Amer. Math. Soc. 52, p.853-858} (2005)

\bibitem [M-W] {M-W} \sc   C.E. Mueller, F.B. Weissler \emph{Hypercontractivity for the heat semigroup for ultraspherical polynomials and on the $n$-Spheres} \rm{Journ. of Funct. Analysis Vol 48, 252-283} (1982)

\bibitem [Nee] {Nee} \sc   J. Neeman, \emph{Isoperimetry and noise sensitivity in Gaussian space} \rm{Ph.Dissertation, University of California, Berkeley} (2013)

\bibitem [R-R] {R-R} \sc   M. Ritor\'e, C. Rosales  \emph{Existence and characterisation of regions minimizing perimeter under a volume constraint inside Euclidean cones} \rm{Tran. Amer. Math. Sco., electronics, 4601-4622} (2004)

\bibitem [Rot] {Rot} \sc O. Rothaus \emph{Hypercontractivity and the Bakry-Emery criterion for compact Lie groups} \rm {J. Funct. Anal. 65, no. 3, 358–367} (1986)

\bibitem [W-W] {W-W} \sc G. Wei, W. Wylie \emph{Comparison Geometry for the Bakry-Emery Ricci Tensor} \rm {J. Diff. Geom. 83, no. 2, 337–405} (2009)

\end{thebibliography}
\end{document}